\documentclass[onefignum,onetabnum]{siamart190516}
\usepackage{amsmath, amssymb}
\usepackage{cleveref}
\usepackage{graphicx}
\usepackage{enumerate}

\def\l{\left}
\def\r{\right}

\def\mR{\mathbb{R}}
\def\mRd{\mathbb{R}^d}
\def\mRdp{\mathbb{R}^{d+1}}
\def\vp{\varphi}
\def\div{\text{div}}
\def\PV{\text{P.V.}}
\def\Gts{\widetilde{G}_s}
\def\Th{\mathcal{T}_h}
\def\Nhi{\mathcal{N}_h^\circ}
\def\Nhc{\mathcal{N}_h^c}
\def\interp{\mathcal I_h}
\def\NOmega{\mathcal{N}_{\Omega}}
\def\NLam{\mathcal{N}_{\Lambda}}

\def\Cl{\Pi_h}
\def\ve{\varepsilon}
\def\pO{\partial\Omega}

\def\x{\texttt{x}}
\newcommand{\phii}{\varphi}
\newcommand{\pp}{\partial}

\newcommand{\eps}{\varepsilon}
\def\dist{\textrm{dist}}

\newtheorem{Theorem}{Theorem}[section]
\newtheorem{Lemma}[Theorem]{Lemma}
\newtheorem{Proposition}[Theorem]{Proposition}
\newtheorem{Corollary}[Theorem]{Corollary}
\newtheorem{Remark}[Theorem]{Remark}

\usepackage{algpseudocode,algorithm,algorithmicx}

\algrenewcommand\algorithmicrequire{\textbf{Input:}}
\algrenewcommand\algorithmicensure{\textbf{Output:}}

\numberwithin{equation}{section}

\begin{document}

\title{Finite element algorithms for \\ nonlocal
minimal graphs}

\author{Juan Pablo~Borthagaray\thanks{Departamento de Matem\'atica y Estad\'istica del Litoral, Universidad de la Rep\'ublica, Salto, Uruguay
		(\email{jpborthagaray@unorte.edu.uy}). JPB has been supported in part by NSF grant DMS-1411808 and Fondo Vaz Ferreira grant 2019-068.}
	\and Wenbo~Li \thanks{Department of Mathematics, University of Tennessee, Knoxville, TN 37996, USA
		(\email{wli50@utk.edu}). WL has been supported in part by NSF grant DMS-1411808 and the Patrick and Marguerite Sung Fellowship in Mathematics of the University of Maryland.}
	\and Ricardo H.~Nochetto \thanks{Department of Mathematics and Institute for Physical Science and Technology, University of Maryland, College Park, MD 20742, USA (\email{rhn@umd.edu}). RHN has been supported in part by NSF grants DMS-1411808 and DMS-1908267.}
}

\headers{FEMs for nonlocal minimal graphs}{J.P. Borthagaray, W. Li, AND R.H. Nochetto}

\maketitle


\begin{abstract}
We discuss computational and qualitative aspects of the fractional Plateau and the prescribed fractional mean curvature problems on bounded domains subject to exterior data being a subgraph. We recast these problems in terms of energy minimization, and we discretize the latter with piecewise linear finite elements. 
For the computation of the discrete solutions, we propose and study a gradient flow and a Newton scheme, and we quantify the effect of Dirichlet data truncation. 
We also present a wide variety of numerical experiments that illustrate qualitative and quantitative features of fractional minimal graphs and the associated discrete problems.
\end{abstract}

\begin{keywords}
	nonlocal minimal surfaces, finite elements, fractional diffusion
\end{keywords}

\begin{AMS}
	49Q05, 35R11, 65N12, 65N30
\end{AMS}

\section{Introduction} \label{sec:intro}
This paper is the continuation of \cite{BoLiNo19analysis}, where the authors proposed and analyzed a finite element scheme for the computation of fractional minimal graphs of order $s \in (0,1/2)$ over bounded domains. That problem can be interpreted as a nonhomogeneous Dirichlet problem involving a nonlocal, nonlinear, degenerate operator of order $s + 1/2$. 
In this paper, we discuss computational aspects of such a formulation and perform several numerical experiments illustrating interesting phenomena arising in  fractional Plateau problems and prescribed nonlocal mean curvature problems.

The notion of {\em fractional perimeter} was introduced in the seminal papers by Imbert \cite{Imbert09} and by Caffarelli, Roquejoffre and Savin \cite{CaRoSa10}. These works were motivated by the study of interphases that arise in classical
phase field models when very long space correlations are present. On the one hand, \cite{Imbert09} was motivated by stochastic Ising models with Ka\v{c} potentials with slow decay at infinity, that give rise (after a suitable rescaling) to problems closely related to fractional reaction-diffusion equations such as
\[
\partial_t u_\eps + (-\Delta)^s u_\eps + \frac{f(u_\eps)}{\eps^{1+2s}} = 0,
\]
where $(-\Delta)^s$ denotes the fractional Laplacian of order $s\in (0,1/2)$ and $f$ is a bistable nonlinearity. On the other hand, reference \cite{CaRoSa10} showed that certain threshold dynam-ics-type algorithms, in the spirit of \cite{MeBeOs92} but corresponding to the fractional Laplacian of order $s \in (0,1/2)$ converge (again, after rescaling) to motion by fractional mean curvature.
Fractional minimal sets also arise in the $\Gamma$-limit of nonlocal Ginzburg-Landau energies \cite{SaVa12Gamma}.

We now make the definition of fractional perimeter precise. 
Let $s \in (0,1/2)$ and two sets $A, B \subset \mRd$, $d \ge 1$. Then, the fractional perimeter of order $s$ of $A$ in $B$ is
\[
P_s(A;B) := \frac{1}{2} \iint_{Q_B} \frac{|\chi_A(x) - \chi_A(y)|}{|x-y|^{d+2s}} \, dy dx,
\]
where $Q_B = (\mRd \times \mRd) \setminus (B^c \times B^c)$ and $B^c = \mRd \setminus B$.
Given some set $A_0 \subset \mRd \setminus B$, a natural problem is how to extend $A_0$ into $B$ while minimizing the $s$-perimeter of the resulting set. This is the fractional Plateau problem, and it is known that, if $B$ is a bounded set, it admits a unique solution. Interestingly, in such a case it may happen that either the minimizing set $A$ is either empty in $B$ or that it completely fills $B$. This is known as a stickiness phenomenon \cite{DipiSavinVald17}.

In this work, we analyze finite element methods to compute fractional minimal graphs on bounded domains. Thus, we consider $s$-minimal sets on a cylinder $B = \Omega \times \mR$, where $\Omega$ is a bounded and sufficiently smooth domain, with exterior data being a subgraph,
\begin{equation*}
A_0 = \l\{ (x', x_{d+1}) \colon x_{d+1} < g(x'), \; x' \in \mRd \setminus \Omega \r\},
\end{equation*}
for some continuous function $g: \mRd \setminus \Omega \to \mR$. We briefly remark some key features of this problem:

\begin{enumerate}[$\bullet$]
\item A technical difficulty arises immediately: all sets $A$ that coincide with $A_0$ in $\mRdp \setminus B$ have infinite $s$-perimeter in $B$. To remedy this issue, one needs to introduce the notion of {\em locally} minimal sets \cite{Lomb16Approx}.

\item There exists a unique locally $s$-minimal set, and it is given by the subgraph of a certain function $u$, cf. \cite{DipiSavinVald16Graph, Lombardini-thesis}. Thus, one can restrict the minimization problem to the class of subgraphs of functions that coincide with $g$ on $\Omega^c$.

\item If the exterior datum $g$ is a bounded function, then one can replace the infinite cylinder $B = \Omega \times \mR$ by a truncated cylinder $B_M = \Omega \times (-M,M)$ for some $M>0$ sufficiently large \cite[Proposition 2.5]{Lombardini-thesis}.

\item Let $A$ be the subgraph of a certain function $v$ that coincides with $g$ on $\Omega^c$. One can rewrite $P_s(A,B_M)$ as
\[
P_s(A,B_M) = I_s[v] + C(M,d,s,\Omega,g),
\]
where $I_s$ is the nonlocal energy functional defined in \eqref{E:NMS-Energy-Graph} below \cite[Proposition 4.2.8]{Lombardini-thesis}, \cite[Proposition 2.3]{BoLiNo19analysis}.
\end{enumerate}

Therefore, an equivalent formulation to the Plateau problem for nonlocal minimal graphs consists in finding a function $u: \mRd \to \mR$, with the constraint $u = g$ in ${\Omega}^c$, such that it minimizes the strictly convex energy
\begin{equation}\label{E:NMS-Energy-Graph}
I_s[u] := \iint_{Q_{\Omega}} F_s\l(\frac{u(x)-u(y)}{|x-y|}\r) \frac{1}{|x-y|^{d+2s-1}} \;dxdy,
\end{equation}
where $F_s$ is defined as 
\begin{equation} \label{E:def_Fs}
F_s(\rho) := \int_0^\rho \frac{\rho-r}{\l( 1+r^2\r)^{(d+1+2s)/2}} dr.
\end{equation}

A remarkable difference between nonlocal minimal surface problems and their local counterparts is the emergence of {\em stickiness} phenomena \cite{DipiSavinVald17}. In the setting of this paper, this means that the minimizer may be discontinuous across $\pp\Omega$. As shown by Dipierro, Savin and Valdinoci \cite{DipiSavinVald19nonlocal}, stickiness is indeed the typical behavior of nonlocal minimal graphs in case $\Omega \subset \mR$. When $\Omega \subset \mR^2$, reference \cite{DipiSavinVald19boundary} proves that, at any boundary points at which stickiness does not happen, the tangent planes of the traces from the interior necessarily coincide with those of the exterior datum. Such a hard geometric constraint is in sharp contrast with the case of classical minimal graphs.
In spite of their boundary behavior, fractional minimal graphs are smooth in the interior of the domain. Indeed, with the notation and assumptions from above it holds that $u \in C^\infty (\Omega)$; see \cite[Theorem 1.1]{CabreCozzi2017gradient}, and \cite{Barrios2014bootstrap, Figalli2017regularity}.

Our previous work \cite{BoLiNo19analysis} introduced and studied a finite element scheme for the computation of fractional minimal graphs. We proved convergence of the discrete minimizers as the mesh size tends to zero, both in suitable Sobolev norms and with respect to a novel geometric notion of error \cite{BoLiNo19analysis}. Stickiness phenomena was apparent in the experiments displayed in \cite{BoLiNo19analysis}, even though the finite element spaces consisted of continuous, piecewise linear functions.
We also refer the reader to \cite{BoLiNo19linear} for further numerical examples and discussion on computational aspects of fractional minimal graph problems.

This paper is organized as follows. \Cref{sec:formulation} gives the formulation of the minimization problem we aim to solve, and compares it with the classical minimal graph problem. Afterwards, in \Cref{sec:discretization} we introduce our finite element method and review theoretical results from \cite{BoLiNo19analysis} regarding its convergence. \Cref{sec:schemes} discusses computational aspects of the discrete problem, including the evaluation of the nonlocal form that it gives rise to, and the solution of the resulting discrete nonlinear equation via a semi-implicit gradient flow and a damped Newton method. Because the Dirichlet data may have unbounded support, we discuss the effect of data truncation and derive explicit bounds on the error decay with respect to the diameter of the computational domain in \Cref{sec:support_g}. \Cref{sec:prescribed_curvature} is concerned with the prescribed nonlocal mean curvature problem. Finally, \Cref{sec:experiments} presents a number of computational experiments that explore qualitative and quantitative features of nonlocal minimal graphs and functions of prescribed fractional mean curvature, the conditioning of the discrete problems and the effect of exterior data truncation.

\section{Formulation of the problem} \label{sec:formulation}
We now specify the problem we aim to solve in this paper and pose its variational formulation. Let $s \in (0,1/2)$ and $g \in L^\infty(\Omega^c)$ be given. 
We consider the space
\begin{equation*}
\mathbb{V}^g := \{ v \colon \mRd \to \mR \; \colon \; v\big|_\Omega \in W^{2s}_1(\Omega), \ v = g \text{ in } {\Omega}^c \}, 
\end{equation*}
equipped with the norm
\[
\| v \|_{\mathbb{V}^g } := \| v \|_{L^1(\Omega)} + | v |_{\mathbb{V}^g },
\]
where
\[
| v |_{\mathbb{V}^g} := \iint_{Q_{\Omega}}  \frac{|v(x)-v(y)|}{|x-y|^{d+2s}} dxdy , 
\]
where $Q_{\Omega} = (\mathbb{R}^d\times\mathbb{R}^d) \setminus (\Omega\times\Omega)$.
The space $\mathbb{V}^g$ can be understood as that of functions in $W^{2s}_1 (\Omega)$ with `boundary value' $g$. The seminorm in $\mathbb{V}^g$ does not take into account interactions over $\Omega^c \times \Omega^c$, because these are fixed for the class of functions we consider; therefore, we do not need to assume $g$ to be a function in $W^{2s}_1(\Omega^c)$. In particular, $g$ may not decay at infinity.
In case $g$ is the zero function, the space $\mathbb{V}^g$ coincides with the standard zero-extension Sobolev space $\widetilde{W}^{2s}_1(\Omega)$; for consistency of notation, we denote such a space by $\mathbb{V}^0$.

For convenience, we introduce the following notation: given a function $u \in \mathbb{V}^g$, the form $a_u \colon \mathbb{V}^g \times \mathbb{V}^0 \to \mathbb{R}$ is
\begin{equation} \label{E:def-a}
a_u(w,v) := \iint_{Q_{\Omega}} \Gts\l(\frac{u(x)-u(y)}{|x-y|}\r) \frac{(w(x)-w(y))(v(x)-v(y))}{|x-y|^{d+1+2s}}dx dy, 
\end{equation}
where 
\begin{equation} \label{eq:Gts}
\Gts(\rho) := \int_0^1 (1+ \rho^2 r^2)^{-(d+1+2s)/2} dr.
\end{equation}
It is worth noticing that $\Gts(\rho) \to 0$ as $|\rho| \to \infty$. Thus, the weight in \eqref{E:def-a} degenerates whenever the difference quotient $\frac{|u(x)-u(y)|}{|x-y|}$ blows up.

The weak formulation of the fractional minimal graph problem can be obtained by the taking first variation of $I_s[u]$ in \eqref{E:NMS-Energy-Graph} in the direction $v$. As described in \cite{BoLiNo19analysis}, that problem reads: find $u \in \mathbb{V}^g$ such that
\begin{equation}\label{E:WeakForm-NMS-Graph}
a_u(u,v) = 0 \quad \forall v \in \mathbb{V}^0.
\end{equation}

In light of the previous considerations, equation \eqref{E:WeakForm-NMS-Graph} can be regarded as a fractional diffusion problem of order $s+1/2$ in $\mathbb{R}^d$ with weights depending on the solution $u$ and fixed nonhomogeneous boundary data $g$.

\begin{Remark}[comparison with local problems]
Roughly, in the classical minimal graph problem, given some boundary data $g$, one seeks for a function $u \in g + H^1_0(\Omega)$ such that 
\[
\int_\Omega \frac{1}{\sqrt{1+|\nabla u|^2}} \, \nabla u \cdot \nabla v \; dx = 0, \quad \forall v \in H^1_0(\Omega).
\]
The integral above can be interpreted as a weighted $H^1$-form, where the weight depends on $u$ and degenerates as $|\nabla u| \to \infty$. 

In a similar way, problem \eqref{E:WeakForm-NMS-Graph} involves a weighted $H^{s+1/2}$-form, in which the weight depends on $u$ and degenerates as $\frac{|u(x)-u(y)|}{|x-y|} \to \infty$. In this sense, it is not surprising that the fractional-order problems converge to the local ones as $s \to 1/2$. We refer to \cite[Section 5]{BoLiNo19analysis} for a discussion on this matter.
\end{Remark}

\section{Finite element discretization}  \label{sec:discretization}

In this section we first introduce the finite element spaces and the discrete formulation of problem \eqref{E:WeakForm-NMS-Graph}. Afterwards, we briefly outline the key ingredients in the convergence analysis for this scheme. For the moment, we shall assume that $g$ has bounded support: 
\[
\mbox{supp}(g) \subset \Lambda, \mbox{for some bounded set } \Lambda.
\]
The discussion of approximations in case of unboundedly supported data is postponed to Section \ref{sec:support_g}.

\subsection{Discrete setting}
We consider a family $\{\Th \}_{h>0}$ of conforming and simplicial triangulations of $\Lambda$, and we assume that all triangulations in $\{\Th \}_{h>0}$ mesh $\Omega$ exactly. Moreover, we assume $\{\Th \}_{h>0}$ to be shape-regular, namely:
\[
  \sigma = \sup_{h>0} \max_{T \in \Th} \frac{h_T}{\rho_T} <\infty,
\]
where $h_T = \mbox{diam}(T)$ and $\rho_T $ is the diameter of the largest ball contained in the element $T\in \Th$. The vertices of $\Th$ will be denoted by $\{\x_i\}$, and the star or patch of $\{ \x_i \}$ is defined as
\[
S_i := \mbox{supp} (\phii_i) ,
\]
where $\phii_i$ is the nodal basis function corresponding to the node $\x_i$.

To impose the condition $u = g$ in $\Omega^c$ at the discrete level, we introduce the exterior interpolation operator
\begin{equation} \label{eq:ex-Clement}
\Cl^c g := \sum_{\x_i \in \Nhc} (\Cl^{\x_i} g) (\x_i) \; \varphi_i,
\end{equation}
where $\Cl^{\x_i} g$ is the $L^2$-projection of  $g\big|_{S_i \cap \Omega^c}$ onto $\mathcal{P}_1(S_i \cap \Omega^c)$. Thus, $\Cl^c g (\x_i)$ coincides with the standard Cl\'ement interpolation of $g$ on $\x_i$ for all nodes $\x_i$ such that $S_i \subset \mRd \setminus \overline{\Omega}$. On the other hand, for nodes $\x_i \in\partial\Omega$, $\Cl^c$ only averages over the elements in $S_i$ that lie in $\Omega^c$.

We consider discrete spaces consisting of piecewise linear functions over $\Th$,
\[
\mathbb{V}_h :=  \{ v \in C(\Lambda) \colon v|_T \in \mathcal{P}_1 \; \forall T \in \Th \}.
\]
To account for the exterior data, we define the discrete counterpart of $\mathbb{V}^g$,
\[
\mathbb{V}_h^g := \{ v \in \mathbb{V}_h \colon \ v|_{\Lambda \setminus \Omega} = \Cl^c g\}.
\]
With the same convention as before, we denote by $\mathbb{V}_h^0$ the corresponding space in case $g \equiv 0$.
Therefore, the discrete weak formulation reads: find $u_h \in \mathbb{V}^g_h$ such that
\begin{equation}\label{E:WeakForm-discrete}
a_{u_h}(u_h, v_h) = 0 \quad \mbox{for all } v_h \in \mathbb{V}^0_h. 
\end{equation}

\begin{Remark}[well-posedness of discrete problem]
Existence and uniqueness of  solutions to the discrete problem \eqref{E:WeakForm-discrete} is an immediate corollary of our assumption $g \in L^{\infty}({\Omega}^c)$. Indeed, from this condition it follows that $u_h$ is a solution of \eqref{E:WeakForm-discrete} if and only if $u_h$ minimizes the strictly convex energy $I_s[u_h]$ over the discrete space $\mathbb{V}_h^g$.
\end{Remark}

\subsection{Convergence}
In \cite{BoLiNo19analysis}, we have proved that solutions to \eqref{E:WeakForm-discrete} converge to the fractional minimal graph as the maximum element diameter tends to $0$.
An important tool in that proof is a quasi-interpolation operator $\interp \colon \mathbb{V}^g \to \mathbb{V}_h^g$ that combines the exterior Cl\'ement interpolation \eqref{eq:ex-Clement} with an interior interpolation operator.  More precisely, we set
\begin{equation} \label{E:interpolation}
\interp v := \Cl^\circ \l(v \big|_\Omega \r)  + \Cl^c g ,
\end{equation}
where $\Cl^\circ$ involves averaging over element stars contained in $\Omega$. Because the minimizer $u$ is smooth in the interior of $\Omega$, but we have no control on its boundary behavior other than the global bound $u \in W^{2s}_1(\Omega)$, we can only assert convergence of the interpolation operator in a $W^{2s}_1$-type seminorm without rates.

\begin{Proposition}[interpolation error] \label{prop:interpolation_estimate}
Let $s \in (0,1/2)$, $\Omega$ be a bounded domain, $g \in C(\Omega^c)$, and $u$ be the solution to \eqref{E:WeakForm-NMS-Graph}. Then, the interpolation operator \eqref{E:interpolation} satisfies
\[
\iint_{Q_{\Omega}}  \frac{\l| (\interp u - u)(x) - (\interp u - u)(y) \r|}{|x-y|^{d+2s}} \; dxdy \to 0 \quad \mbox{ as } h \to 0.
\]
\end{Proposition}

Once we have proved the convergence of $\interp u$ to $u$, energy consistency follows immediately. Since the energy dominates the $W^{2s}_1(\Omega)$-norm \cite[Lemma 2.5]{BoLiNo19analysis}, we can prove convergence in $W^{2r}_1(\Omega)$ for all $r \in [0,s)$ by arguing by compactness.

\begin{Theorem}[convergence] \label{thm:convergence}
Assume $s \in (0,1/2)$, $\Omega \subset \mRd$ is a $C^{1,1}$ domain, and $g \in C_c(\mRd)$. Let $u$ be the minimizer of $I_s$ on $\mathbb{V}^g$ and $u_h$ be the minimizer of $I_s$ on $\mathbb{V}_h^g$. Then, it holds that
\[
\lim_{h \to 0} \| u - u_h \|_{W^{2r}_1(\Omega)} = 0, \quad \forall r \in [0,s).
\]
\end{Theorem}

We finally point out that \cite[Section 5]{BoLiNo19analysis} introduces a geometric notion of error that mimics a weighted $L^2$ discrepancy between the normal vectors to the graph of $u$ and $u_h$. We refer to that paper for further details.

\subsection{Graded meshes} \label{sec:graded}
As mentioned in the introduction, fractional minimal surfaces are smooth in the interior of $\Omega$. The main challenge in their approximation arises from their boundary behavior and concretely, from the genericity of stickiness phenomena, i.e. discontinuity of the solution $u$ across $\partial\Omega$. Thus, it is convenient to a priori adapt meshes to better capture the jump of $u$ near $\pp\Omega$.

In our discretizations, we use the following construction \cite{Grisvard}, that gives rise to shape-regular meshes. Let $h>0$ be a mesh-size parameter and $\mu \ge 1$. Then, we consider meshes $\Th$ such that every element $T \in \Th$ satisfies
\begin{equation}
\label{eq:mesh_grad} 
h_T \approx \left\lbrace \begin{array}{ll}
C(\sigma) h^\mu, & \overline T \cap \pp \Omega \neq \emptyset \\ 
C(\sigma) h \dist(T,\pp\Omega)^{(\mu-1)/\mu}, & \overline T \cap \pp \Omega = \emptyset.
  \end{array} \right.
\end{equation}

These meshes, typically with $\mu=2$, give rise to optimal convergence rates for homogeneous problems involving the fractional Laplacian in 2d \cite{AcosBort2017fractional, BoCi19, BoNoSa18, BdPM}. We point out that in our problem the computational domain strictly contains $\Omega$, because we need to impose the exterior condition $u = g$ on $\Omega^c$.
As shown in \cite{BKP:79, BoNoSa18}, the construction \eqref{eq:mesh_grad} leads to
\[
\dim \mathbb{V}_h^g \approx \left\lbrace \begin{array}{ll}
h^{(1-d)\mu}, & \mu > \frac{d}{d-1}, \\
h^{-d}|\log h|, & \mu  = \frac{d}{d-1}, \\
h^{-d}, & \mu < \frac{d}{d-1}.
  \end{array} \right.
\]

In our applications, because \Cref{thm:convergence} gives no theoretical convergence rates, we are not restricted to the choice $\mu = 2$ in two dimensions: a higher $\mu$ allows a better resolution of stickiness. However, our numerical experiments indicate that the condition number of the resulting matrix at the last step of the Newton iteration deteriorates as $\mu$ increases, cf. \Cref{sec:conditioning}.

\section{Numerical schemes} \label{sec:schemes}
Having at hand a finite element formulation of the nonlocal minimal graph problem and proven its convergence as the mesh size tends to zero, we now address the issue of how to compute discrete minimizers in 1d and in 2d. In first place, we discuss the computation of matrices associated to either the bilinear form $a_{u_h} (\cdot, \cdot)$, or related computations. We propose two schemes for the solution of the nonlinear discrete problems \eqref{E:WeakForm-discrete}: a semi-implicit gradient flow and a damped Newton method. In this section we also discuss the convergence of these two algorithms.

\subsection{Quadrature} \label{sec:quadrature}

We now consider the evaluation of the forms $a_{u_h}(\cdot, \cdot)$ appearing in \eqref{E:WeakForm-discrete}. We point out that, following the implementation techniques from \cite{AcosBersBort2017short, AcosBort2017fractional}, if we are given $u_h \in \mathbb{V}_h^g$ and $v_h \in \mathbb{V}_h^0$, then we can compute $a_{u_h}(u_h, v_h)$.
Indeed, since $a_{u_h}(u_h,v_h)$ is linear in $v_h$ and the latter function can be written in the form $v_h(x) = \sum_{\x_i \in \Nhi} v_i \varphi_i(x)$, we only need to evaluate
\[ \begin{aligned}
a_i & := a_{u_h}(u_h, \varphi_i) \\
 &= \iint_{Q_{\Omega}} \Gts\l(\frac{u_h(x)-u_h(y)}{|x-y|}\r) \frac{(u_h(x)-u_h(y))(\varphi_i(x)-\varphi_i(y))}{|x-y|^{d+1+2s}}dxdy .
\end{aligned} \]
We split $Q_{\Omega} = (\Omega \times \Omega)\cup(\Omega \times \Omega^c)\cup(\Omega^c \times \Omega)$ and, because $\widetilde{G}_s$ is an even function (cf. \eqref{eq:Gts}), we can take advantage that the integrand is symmetric with respect to $x$ and $y$ to obtain
\[ \begin{aligned}
a_i = & \iint_{\Omega \times \Omega} \Gts\l(\frac{u_h(x)-u_h(y)}{|x-y|}\r) \frac{(u_h(x)-u_h(y))(\varphi_i(x)-\varphi_i(y))}{|x-y|^{d+1+2s}}dxdy \\
& + 2 \iint_{\Omega \times \Omega^c} \Gts\l(\frac{u_h(x)-g_h(y)}{|x-y|}\r) \frac{(u_h(x)-g_h(y)) \varphi_i(x)}{|x-y|^{d+1+2s}}dxdy =: a_{i,1} + 2a_{i,2} .
\end{aligned} \]

We assume that the elements are sorted in such a way that the first $\NOmega$ elements mesh $\Omega$, while the remaining $\NLam - \NOmega$ mesh $\Lambda \setminus \Omega$, that is, 
\[
\bigcup_{1 \le i \le \NOmega} \overline{T_i} = \overline{\Omega} \quad  \bigcup_{\NOmega + 1 \le i \le \NLam} \overline{T_i} = \overline{\Lambda} \setminus \Omega.
\]
By doing a loop over the elements of the triangulation, the integrals $a_{i,1}$ and $a_{i,2}$ can be written as:
\[ \begin{aligned}
a_{i,1} = & \sum_{l,m=1}^{\NOmega} \iint_{T_l \times T_m} \Gts\l(\frac{u_h(x)-u_h(y)}{|x-y|}\r) \frac{(u_h(x)-u_h(y))(\varphi_i(x)-\varphi_i(y))}{|x-y|^{d+1+2s}}dxdy ,\\
a_{i,2} = & \sum_{l=1}^{\NOmega}\sum_{m=\NOmega + 1}^{\NLam} \iint_{T_l \times T_m} \Gts\l(\frac{u_h(x)-{g_h(y)}}{|x-y|}\r) \frac{(u_h(x)-{g_h(y)}) \varphi_i(x)}{|x-y|^{d+1+2s}}dxdy \\
& + \sum_{l=1}^{\NOmega} \iint_{T_l \times \Lambda^c}
\Gts\l(\frac{u_h(x)}{|x-y|}\r) \frac{u_h(x)\varphi_i(x)}{|x-y|^{d+1+2s}}dxdy .
\end{aligned} \]
For the double integrals on $T_l \times T_m$ appearing in the definitions of $a_{i,1}$ and $a_{i,2}$, we apply the same type of transformations described in \cite{AcosBersBort2017short,ChernovPetersdorffSchwab:11,SauterSchwab} to convert the integral into an integral over $[0,1]^{2d}$, in which variables can be separated and the singular part can be computed analytically. The integrals over $T_l \times \Lambda^c$ are of the form
\[ 
\int_{T_l} \varphi_i(x) \omega(x) dx ,
\]
where the weight function $\omega$ is defined as
\begin{equation} \begin{split}
    \omega(x) & :=
    \int_{\Lambda^c} G_s\l(\frac{u_h(x)}{|x-y|}\r) \frac{1}{|x-y|^{d+2s}} dy , \\
G_s (\rho) & :=   \int_0^\rho (1+ r^2)^{-(d+1+2s)/2} dr = \rho \, \Gts (\rho). \label{E:DEF-Gs}
\end{split} \end{equation}
Since the only restriction on the set $\Lambda$ is that $\mbox{supp}(g) \subset \Lambda$, without loss of generality we assume that $\Lambda = B_{R}$ is a $d$-dimensional ball with radius $R$. In such a case, the integral over $\Lambda^c$ can be transformed using polar coordinates into:
\[
w(x) = \int_{\partial B_{1}} dS(e) \int_{\rho_0(e,x)}^{\infty} G_s\l(\frac{u_h(x)}{\rho}\r) \rho^{-1-2s} d\rho,
\]
where $\rho_0(e,x)$ is the distance from $x$ to $\partial B_{R}$ in the direction of $e$, which is given by the formula
\[
\rho_0(e,x) = \sqrt{R^2 - |x|^2 + (e \cdot x)^2} - e \cdot x .
\]
The integral over $(\rho_0(e,x),\infty)$ can be transformed to an integral over $(0,1)$ by means of the change of variable $\rho = \rho_0(e,x) \widetilde{\rho}^{-1/(2s)}$, and then approximated by Gaussian quadrature. Combining this approach with suitable quadrature over $\partial B_{1}$ and $T_l$, we numerically compute the integral over $T_l \times \Lambda^c$ for a given $u_h$.

\subsection{Gradient Flow}
Although we can compute $a_{u_h}(u_h, v_h)$ for any given $u_h \in \mathbb{V}_h^g, \ v_h \in \mathbb{V}_h^0$, the nonlinearity of $a_{u_h}(u_h,v_h)$ with respect to $u_h$ still brings difficulties in finding the discrete solution to \eqref{E:WeakForm-discrete}. 
Since $a_{u_h}(u_h,v_h) = \frac{\delta I_s[u_h]}{\delta u_h}(v_h)$ and $u_h$ minimizes the convex functional $I_s[u_h]$ in the space $\mathbb{V}_h^g$, a gradient flow is a feasible approach to solve for the unique minimizer $u_h$.

Given $\alpha \in [0,1)$, and with the convention that $H^0 = L^2$, we first consider a time-continuous $H^\alpha$-gradient flow for $u_h(t)$, namely
\begin{equation}\label{E:Gradient-flow}
\langle \partial_t u_h, v_h \rangle_{H^{\alpha}(\Omega)} = -\frac{\delta I_s}{\delta u_h}(v_h)
    = -a_{u_h}(u_h, v_h), \qquad \forall v_h \in \mathbb{V}^0_h ,
\end{equation}
where $u_h(0) = u_h^0 \in \mathbb{V}^g_h$ (and thus $I_s[u_h^0] < \infty$).
Writing $u_h(t) = \sum_{x_j \in \Nhi} u_j(t) \varphi_i$, local existence and uniqueness of solutions in time for \eqref{E:Gradient-flow} follow from the fact that $a_{u_h}(u_h, \varphi_i)$ is Lipschitz with respect to $u_j$ for any $\varphi_i$. Noticing that the gradient flow \eqref{E:Gradient-flow} satisfies the energy decay property
\[
\frac{d}{dt} I_s[u_h] = \frac{\delta I_s[u_h]}{\delta u_h}(\partial_t u_h) = a_{u_h}(u_h, \partial_t u_h)= -\langle \partial_t u_h, \partial_t u_h \rangle_{H^{\alpha}(\Omega)} \le 0 ,
\]
global existence and uniqueness of solutions in time can also be proved.

Similarly to the classical mean curvature flow of surfaces \cite{Dziuk1999numerical}, there are three standard ways to discretize \eqref{E:Gradient-flow} in time: fully implicit, semi-implicit and fully explicit. Like in the classical case, the fully implicit scheme requires solving a nonlinear equation at every time step, which is not efficient in practice, while the fully explicit scheme is conditionally stable, and hence requires the choice of very small time steps. We thus focus on a {\it semi-implicit} scheme: given the step size $\tau>0$ and iteration counter $k\ge0$, find $u_h^{k+1}\in\mathbb{V}_h^g$ that solves
\begin{equation}\label{E:semi-implicit-GF}
\frac1{\tau} \langle {u^{k+1}_h - u^k_h} \ , \ v_h \rangle_{H^{\alpha}(\Omega)} = -a_{u^k_h}(u^{k+1}_h \ , \ v_h), \qquad \forall v_h \in \mathbb{V}^0_h .
\end{equation} 
The linearity of  $a_{u^k_h}(u^{k+1}_h,v_h)$ with respect to $u^{k+1}_h$ makes \eqref{E:semi-implicit-GF}  amenable for its computational solution.
The following proposition proves the stability of the semi-implicit scheme. Its proof mimics the one of classical mean curvature flow \cite{Dziuk1999numerical}.

\begin{Proposition}[stability of $H^\alpha$-gradient flow]\label{p:stab-semi-implicit}
Assume $u^{k+1}_h, u^k_h \in \mathbb{V}^g_h$ satisfy \eqref{E:semi-implicit-GF}. Then, 
\begin{equation*}\label{E:stab-semi-implicit}
  I_s[u^{k+1}_h] + \frac{1}{\tau} \Vert u^{k+1}_h - u^k_h \Vert^2_{H^{\alpha}(\Omega)}
    \le I_s[u^k_h].
\end{equation*}
\end{Proposition}
\begin{proof}
Choose $v_h = u^{k+1}_h - u^k_h \in \mathbb{V}^0_h$ in \eqref{E:semi-implicit-GF} to obtain
\begin{equation}\label{E:proof-stab-semi-implicit}
\frac{1}{\tau} \Vert u^{k+1}_h - u^k_h\Vert^2_{H^{\alpha}(\Omega)} = -a_{u^k_h}(u^{k+1}_h \, , \, u^{k+1}_h - u^k_h) .
\end{equation}
Next, we claim that for every pair of real numbers $r_0, r_1$, it holds that
\begin{equation}\label{E:stab-prop-Gts}
    (r_1^2 - r_1 r_0) \ \Gts(r_0) \ge F_s(r_1) - F_s(r_0).
\end{equation}
We recall that $F_s$ is defined according to \eqref{E:def_Fs}, that $\Gts$ satisfies $\Gts(r) = \frac{1}{r} G_s(r)$, and that $G_s = F_s'$. Since $F_s$ is a convex and even function, we deduce
\[ 
\begin{aligned}
F_s(r_1) - F_s(r_0) &= F_s(|r_1|) - F_s(|r_0|) \\
&\le F_s(|r_1|) - \left[ F_s(|r_1|) + (|r_0| - |r_1|) \ G_s(|r_1|) \right] \\
&= (|r_1| - |r_0|) \ |r_1| \ \Gts(|r_1|).
\end{aligned} 
\]
We add and subtract $(|r_1| - |r_0|) \, |r_1| \, \Gts(|r_0|)$ above and use that $\Gts$ is even, decreasing on $[0,\infty)$ and non-negative, to obtain 
\[ 
\begin{aligned}
F_s(r_1) - F_s(r_0) 
& \le (|r_1| - |r_0|) \ |r_1| \ \Gts(|r_0|) + |r_1| \ (|r_1| - |r_0|) \l( \Gts(|r_1|) - \Gts(|r_0|) \r) \\
&\le (|r_1| - |r_0|) \ |r_1| \ \Gts(|r_0|) \\
&= (r_1^2 - |r_0|\ |r_1|) \ \Gts(|r_0|) \\
&\le (r_1^2 - r_0 r_1) \ \Gts(r_0) .
\end{aligned} 
\]
This proves \eqref{E:stab-prop-Gts}. Finally, define $d_k(x,y) := \frac{u_h^{k}(x) - u_h^{k}(y)}{|x-y|}$ and  set $r_0 = d_k$ and $r_1 = d_{k+1}$ in \eqref{E:stab-prop-Gts} to deduce that
\[ \begin{aligned}
a_{u^k_h}(u^{k+1}_h \ , \ u^{k+1}_h - u^k_h) &= \iint_{Q_{\Omega}} \Gts\l( d_k(x,y) \r) \frac{d_{k+1}(x,y) (d_{k+1}(x,y) - d_{k}(x,y)) }{|x-y|^{d-1+2s}} dxdy \\
&\ge \iint_{Q_{\Omega}} \frac{F_s(d_{k+1}(x,y)) - F_s(d_{k}(x,y))}{|x-y|^{d-1+2s}} dxdy \\
&= I_s[u^{k+1}_h] - I_s[u^{k}_h] .
\end{aligned} \]
Combining this with \eqref{E:proof-stab-semi-implicit} finishes the proof.
\end{proof}

Upon writing $w_h^k := u_h^{k+1} - u_k$, the semi-implicit scheme \eqref{E:semi-implicit-GF} becomes \eqref{E:semi-implicit-wh}, which is the crucial step of \Cref{alg:semi-implicit-GF} to solve \eqref{E:WeakForm-discrete}.
\begin{algorithm}
  \caption{Semi-implicit gradient flow}
    \label{alg:semi-implicit-GF}
  \begin{algorithmic}[1] 
    \State Select an arbitrary initial $u_h^0 \in \mathbb{V}^g_h$, let $k = 0$, and set $\| w_h^0 \|_{H^\alpha(\Omega)} = \texttt{Inf}$. Choose a time step
    $\tau > 0$ and a small number $\ve > 0$.
    \While{ 
    $\| w_h^k \|_{H^\alpha(\Omega)} > \ve$
    }
    
        \State Find $w_h^{k+1} \in \mathbb{V}^0_h$ such that
\begin{equation}\label{E:semi-implicit-wh}
\langle w_h^{k+1}, v_h \rangle_{H^{\alpha}(\Omega)} + \tau a_{u_h^k}(w_h^{k+1}, v_h)
= -a_{u_h^k}(u_h^k, v_h) \ , \quad \forall v_h \in \mathbb{V}^0_h.
\end{equation}
        \State Set $u_h^{k+1} = u_h^k + \tau \; w_h^{k+1}$ and $k = k+1$.
    \EndWhile
  \end{algorithmic}
\end{algorithm}
Equation \eqref{E:semi-implicit-wh} boils down to solving the linear system $\left(M + \tau K^{k} \right) W^k = F^{k}$.
In case $\alpha = 0$, the matrix $M = (M_{ij})$  is just a mass matrix, while if $\alpha>0$, $M$ is the stiffness matrix for the linear fractional diffusion problem of order $\alpha$, given by
\[
M_{ij} := \iint_{Q_{\Omega}} \frac{(\varphi_i(x)-\varphi_i(y))(\varphi_j(x)-\varphi_j(y))}{|x-y|^{d+2\alpha}}dxdy  \quad (\alpha > 0).  
\]
The matrix $K^{k} = \left(K^{k}_{ij}\right)$ is the stiffness matrix for a weighted linear fractional diffusion of order $s + \frac{1}{2}$, whose elements $K^{k}_{ij} := a_{u_h^k} (\varphi_i, \varphi_j)$ are given by
\[
K^{k}_{ij} = \iint_{Q_{\Omega}} \Gts\l(\frac{u_h^k(x)-u_h^k(y)}{|x-y|}\r) \frac{(\varphi_i(x)-\varphi_i(y))(\varphi_j(x)-\varphi_j(y))}{|x-y|^{d+1+2s}}dxdy ,
\]
and can be computed as described in \Cref{sec:quadrature}. The right hand side vector is $F^{k} = - K^k U^k$, where $U^k = \left(U^{k}_i \right)$ is the vector $U^k_i = u_h^k(x_i)$, i.e., $f^{k}_i = -a_{u_h^k}(u_h^k, \varphi_i)$.

Because of \Cref{p:stab-semi-implicit} (stability of $H^\alpha$-gradient flow), the loop in \Cref{alg:semi-implicit-GF} terminates in finite steps. Moreover, {using the continuity of $a_{u_h^k}(\cdot,\cdot)$ in $[H^{\frac12 + s}(\Omega)]^2$, which is uniform in $u_h^k$, together with an inverse estimate and $0\le \alpha\le \frac12 + s$ gives
\[
\big| a_{u_h^k}(w_h^{k+1},v_h) \big| \lesssim
|w_h^{k+1}|_{H^{\frac12 + s}(\Omega)} |v_h|_{H^{\frac12 + s}(\Omega)} \lesssim
h_{\text{min}}^{-1-2s+2\alpha} |w_h^{k+1}|_{H^{\alpha}(\Omega)} |v_h|_{H^{\alpha}(\Omega)},
\]  
where the hidden constant depends on the mesh shape-regularity and $h_{\text{min}}$ is the minimum element size. Therefore, the last iterate $u_h^k$ of \Cref{alg:semi-implicit-GF} satisfies the residual estimate
\[
\max_{v_h\in\mathbb{V}^0_h} \frac{\big|a_{u_h^k}(u_h^k, v_h)\big|}{\|v_h\|_{H^{\alpha}(\Omega)}}
\lesssim \ve \Big( 1 + \tau h_{\text{min}}^{-1-2s+2\alpha} \Big) .
\]

\subsection{Damped Newton algorithm} \label{sec:Newton}

Since the semi-implicit gradient flow is a first order method to find the minimizer of the discrete energy, it may converge slowly in practice. Therefore, it is worth
having an alternative algorithm to solve \eqref{E:WeakForm-discrete} faster. With that goal in mind, we present in the following a damped Newton scheme, which is a second order method and thus improves the speed of computation.

\begin{algorithm}
  \caption{Damped Newton Algorithm}
    \label{alg:Damped-Newton}
  \begin{algorithmic}[1] 
    \State Select an arbitrary initial $u_h^0 \in \mathbb{V}^g_h$ and let $k = 0$. Choose 
    a small number $\ve > 0$.
    \While{ $ \Vert \{ a(u_h^k, \vp_i) \}_{i=1}^m \Vert_{l^2} > \ve$ }
        \State Find $w_h^k \in \mathbb{V}^0_h$ such that
\begin{equation}\label{E:damped-Newton-wh}
\frac{\delta a_{u_h}(u_h^k, v_h)}{\delta u_h^k}(w_h^k)
= -a_{u_h^k}(u_h^k, v_h), \qquad \forall v_h \in \mathbb{V}^0_h.
\end{equation}
        \State Determine the minimum $n \in \mathbb{N}$ such that
        $u_h^{k,n} := u_h^k + 2^{-n} w_h^k$ satisfies
\begin{equation*}
\Vert \{ a_{u_h^k}(u_h^{k,n}, \vp_i) \}_{i=1}^m \Vert_{l^2} \leq (1 - 2^{-n-1})
\Vert \{ a_{u_h^k}(u_h^k, \vp_i) \}_{i=1}^m \Vert_{l^2}
\end{equation*}
        \State Let $u_h^{k+1} = u_h^{k,n} $ and $k = k+1$.
    \EndWhile
  \end{algorithmic}
\end{algorithm}

To compute the first variation of $a_u(u,v)$ in \eqref{E:def-a} with respect to $u$, which is also the second variation of $I_s[u]$, we make use of $r\Gts(s)= G_s(r)$ and obtain
\begin{equation*}\label{E:second-variation-I_s}
    \frac{\delta a_u(u,v)}{\delta u}(w) = \iint_{Q_{\Omega}} G'_s \l(\frac{u(x)-u(y)}{|x-y|}\r) \frac{(w(x)-w(y))(v(x)-v(y))}{|x-y|^{d+1+2s}}dxdy.
\end{equation*}
The identity $G'_s(a) = (1+a^2)^{-(d+1+2s)/2}$ can be easily determined from \eqref{E:DEF-Gs}. Even though this first variation is not well-defined for an arbitrary $u \in \mathbb{V}^g$ and $v,w \in \mathbb{V}^0$, its discrete counterpart $\frac{\delta a_{u_h}(u_h,v_h)}{\delta u_h}(w_h)$ is well-defined for all $u_h \in \mathbb{V}_h^g$, $v_h,w_h \in \mathbb{V}_h^0$ because they are Lipschitz. Our damped Newton algorithm for \eqref{E:WeakForm-discrete} is presented in \Cref{alg:Damped-Newton}. 

\begin{Lemma}[convergence of \Cref{alg:Damped-Newton}]\label{L:conv-Newton}
The iterates $u_h^k$ of \Cref{alg:Damped-Newton} converge quadratically to the unique solution of \eqref{E:WeakForm-discrete} from any initial condition.
\end{Lemma}
\begin{proof}
Since $I_s[u_h]$ is strictly convex, the convergence of $u_h^k$ to the solution of discrete problem \eqref{E:WeakForm-discrete} is guaranteed by the theory of numerical optimization in finite dimensional spaces (see \cite{kelley99}, for example).
\end{proof}

The critical step in \Cref{alg:Damped-Newton} is to solve the equation \eqref{E:damped-Newton-wh}. Due to the linearity of $\frac{\delta a_{u_h^k}(u_h^k, v_h)}{\delta u_h^k}(w_h^k)$ with respect to $v_h$ and $w_h^k$, we just need to solve a linear system $\widetilde{K}^k W^k = F^k$, where the right hand side $F^k = (f^k_i)$ is the same as the one in solving \eqref{E:semi-implicit-wh}, namely, $f^k_i = a_{u_h^k}(u_h^k,\varphi_i)$. The matrix $\widetilde{K}^k = (\widetilde{K}^k_{ij})$, given by
\[
\widetilde{K}^k_{ij} = \iint_{Q_{\Omega}} G'_s\l(\frac{u_h^k(x)-u_h^k(y)}{|x-y|}\r) \frac{(\varphi_i(x)-\varphi_i(y))(\varphi_j(x)-\varphi_j(y))}{|x-y|^{d+1+2s}}dxdy ,
\]
is the stiffness matrix for a weighted linear fractional diffusion of order $s + \frac{1}{2}$. Since the only difference with the semi-implicit gradient flow algorithm is the weight, the elements in $\widetilde{K}^k$ can be computed by using the same techniques as for $K^k$.

\section{Unboundedly supported data} \label{sec:support_g}
Thus far, we have taken for granted that $g$ has bounded support, and that the computational domain covers $\mbox{supp}(g)$. We point out that most of the theoretical estimates only require $g$ to be locally bounded. Naturally, in case $g$ does not have compact support, one could simply multiply $g$ by a cutoff function and consider discretizations using this truncated exterior condition. Here we quantify the consistency error arising in this approach. More precisely, given $H>0$, we consider $\Omega_H$ to be a bounded open domain containing $\Omega$ and such that $d(x, \overline\Omega) \simeq H$ for all $x \in \partial \Omega_H$,  
and choose a cutoff function $\eta_H \in C^\infty(\Omega^c)$ satisfying
\begin{equation*} \label{eq:cutoff}
0\le \eta_H \le 1, \quad  \text{supp}(\eta_H)\subset \overline{\Omega}_{H+1} \setminus \Omega,  \quad \eta_H(x)=1 \quad \mbox{in} \quad \Omega_{H} \setminus \Omega   .
\end{equation*}

We replace $g$ by $g_H := g \eta_H$, and consider problem \eqref{E:WeakForm-NMS-Graph} using $g_H$ as Dirichlet condition. Let $u^H \in \mathbb{V}^{g_H}$ be the solution of such a problem, and $u_h^H$ be the solution of its discrete counterpart over a certain mesh with element size $h$. Because of \Cref{thm:convergence} we know that, for all $r \in[0,s)$,
\[
u_h^H \to u^H \quad \mbox{in } W^{2r}_1(\Omega) \quad \mbox{as } h \to 0.
\]
Therefore we only need to show that, in turn, the minimizers of the truncated problems satisfy $u^H \to u$ as $H \to \infty$ in the same norm.
As a first step, we compare the differences in the energy between truncated and extended functions. For that purpose, we define the following truncation and extension operators:
\[ \begin{array}{ll}
T_H \colon \mathbb{V}^g \to \mathbb{V}^{g_H}, \quad & T_H v = v \eta_H , \\ 
E_H \colon \mathbb{V}^{g_H} \to \mathbb{V}^g, \quad & E_H w = w + (1 - \eta_H) g . 
\end{array} \]

\begin{Proposition}[truncation and extension] The following estimates hold for every $v \in \mathbb{V}^g \cap L^\infty(\mRd)$, and $w \in \mathbb{V}^{g_H}\cap L^\infty(\mRd)$:
\[ \begin{aligned}
|I_s[v] - I_s[T_H v]| \lesssim H^{-1-2s}, \\
|I_s[w] - I_s[E_H w]| \lesssim H^{-1-2s}.
\end{aligned} \]
\end{Proposition}
\begin{proof}
We prove only the first estimate, as the second one follows in the same fashion. Because $v = T_H v$ in $\Omega_{H}$, we have
\[ \begin{aligned}
|I_s[v] & -  I_s[T_H v]| \\
& \le 2 \int_\Omega \int_{\Omega_H^c} \l| F_s\l(\frac{v(x)-v(y)}{|x-y|}\r) - F_s\l(\frac{v(x)- T_Hv(y)}{|x-y|}\r) \r| \frac{1}{|x-y|^{d+2s-1}} \; dydx .
\end{aligned} \]

From definition \eqref{E:def_Fs}, it follows immediately that $F_s(0) = F'_s(0) = 0$, and thus $F_s (\rho) \le C \rho^2$ if $\rho \lesssim 1$.
Combining this with the fact that $|v(x) - v(y)| \le 2 \| v \|_{L^\infty(\mRd)}$ and $|v(x) - T_H v(y)| \le 2 \| v \|_{L^\infty(\mRd)}$ for a.e. $x\in \Omega, y \in \Omega^c$, and integrating in polar coordinates, we conclude
\[ \begin{aligned}
|I_s[v] - I_s[T_H v]| & \lesssim \|v\|_{L^\infty(\Omega^c)}^2 \int_\Omega \int_{\Omega_H^c} \frac{1}{|x-y|^{d+2s+1}} \;dxdy  \lesssim  H^{-1-2s}.
\end{aligned} \]
This concludes the proof.
\end{proof}

The previous result leads immediately to an energy consistency estimate for the truncated problem. 

\begin{Corollary}[energy consistency] \label{cor:consistency_truncation}
The minimizers of the original and truncated problem satisfy 
\[
\l|I_s[u] - I_s[u^H]\r| \lesssim H^{-1-2s}.
\]
\end{Corollary}
\begin{proof}
Since $u^H$ is the minimizer over $\mathbb{V}^{g_H}$ and $T_H u\in\mathbb{V}^{g_H}$,
we deduce
\[
I_s[u^H] - I_s[u] \le I_s[T_H u] - I_s[u] \lesssim H^{-1-2s}.
\]
Conversely, using that $u$ is the minimzer over $\mathbb{V}^{g}$ and $Eu^H\in\mathbb{V}^{g}$, we obtain
\[
I_s[u] - I_s[u^H] \le I_s [Eu^H] - I_s [u^H] \lesssim H^{-1-2s},
\]
and thus conclude the proof.
\end{proof}

The energy $I_s$ is closely related to the $W^{2s}_1(\Omega)$-norm, in the sense that one is finite if and only if the other one is finite \cite[Lemma 2.5]{BoLiNo19analysis}. Thus, in the same way as in \Cref{thm:convergence} (convergence), energy consistency yields convergence in $W^{2r}_1(\Omega)$ for all $r \in [0,s)$.

\begin{Proposition}[convergence]
Let $u$ and $u_H$ be minimizers of $I_s$ over $\mathbb{V}^{g}$ and $\mathbb{V}^{g_H}$, respectively. Then for all $r \in [0,s)$, it holds that
\[
\lim_{H \to \infty} \|u - u^H \|_{W^{2r}_1(\Omega)} = 0 .
\]
\end{Proposition}
\begin{proof}
The proof proceeds using the same arguments as in \cite[Theorem 4.3]{BoLiNo19analysis}. In fact, from \Cref{cor:consistency_truncation} we deduce that $\{ I_s[u^H] \}$ is uniformly bounded and therefore $\{u^H\}$ is bounded in $W^{2s}_1(\Omega)$. It follows that, up to a subsequence, $u^H$ converges in $L^1(\Omega)$ to a limit $\widetilde u$. Also, because $u^H = g$ in $\Omega_H$, we can extend $\widetilde u$ by $g$ on $\Omega^c$, and have $u^H \to u$ a.e in $\mRd$. We then can invoke Fatou's lemma and \Cref{cor:consistency_truncation} to deduce that
\[
I_s [\widetilde u ] \le \liminf_{H \to \infty} I_s [ u^H ] \lesssim \liminf_{H \to \infty} I_s [u] + H^{-1-2s} = I_s [u].
\]
Because $\widetilde u \in \mathbb{V}^g$, we deduce that $\widetilde u = u$ whence $u_H \to u$ in $L^1(\Omega)$ as $H\to0$. By interpolation, we conclude that convergence in $W^{2r}_1(\Omega)$ holds for all $r \in [0,s)$.
\end{proof}

\section{Prescribed nonlocal mean curvature} \label{sec:prescribed_curvature}

In this section, we briefly introduce the problem of computing graphs with prescribed nonlocal mean curvature. More specifically, we address  the computation of a function $u$ such that for a.e. $x \in \Omega$, a certain nonlocal mean curvature at $\big(x,u(x)\big)$ is equal to a given function $f(x)$. For a set $E \subset \mRdp$ and $\widetilde{x} \in \partial E$, such nonlocal mean curvature operator is defined as \cite{CaRoSa10}
\begin{equation*}\label{E:NMS-def-NonLocalCurv}
H_s[E](\widetilde{x}) := \PV \int_{\mRdp} \frac{\chi_{E^c}(\widetilde{y}) - \chi_{E}(\widetilde{y})}{|\widetilde{x}-\widetilde{y}|^{d+1+2s}} d\widetilde{y}.
\end{equation*}
In turn,  for $\widetilde{x} = (x,u(x))$ on the graph of $u$, this 
can be written as \cite[Chapter 4]{Lombardini-thesis}
\[ \begin{aligned}
H_s[u](x) = \PV \int_{\mR^d} G_s\l(\frac{u(x)-u(y)}{|x-y|}\r) \frac{dy}{|x-y|^{d+2s}}.
\end{aligned} \] 

To recover the classical mean curvature in the limit $s \to \frac12^-$, it is necessary to normalize the operator $H_s$ accordingly. Let $\alpha_{d}$ denote the volume of the $d$-dimensional unit ball, and consider the prescribed nonlocal mean curvature problem
\begin{equation}\label{E:NMS-Prescribe-Nmc}
\l\{\begin{array}{rl}
\frac{1-2s}{d \alpha_d}H_s[u](x) = f(x), & x \in \Omega, \\
u(x) = g(x),  &  x \in \mRd \setminus \Omega.
\end{array}\r.
\end{equation}
The scaling factor $\frac{1-2s}{d \alpha_d}$ yields \cite[Lemma 5.8]{BoLiNo19analysis}
\begin{equation}\label{E:NMS-NonLocalCurv-asymp}
\lim_{s \to \frac12^-}\frac{1-2s}{d \alpha_d}H_s[E](x) = H[E](x),
\end{equation}
where $H[E]$ denotes the classical mean curvature operator.
Therefore, in the limit $s \to \frac12^-$, formula \eqref{E:NMS-Prescribe-Nmc} formally becomes the following Dirichlet problem for graphs of prescribed classical mean curvature:
\begin{equation}\label{E:NMS-Prescribe-MC}
\l\{\begin{array}{rl}
\frac{1}{d} \; \div\Big( \frac{\nabla u (x)}{\big(1+|\nabla u (x)|^2 \big)^{1/2}} \Big) = f(x), & x \in \Omega, \\
u(x) = g(x), &  x \in \pO.
\end{array}\r.
\end{equation}

An alternative formulation of the prescribed nonlocal mean curvature problem for graphs is to find $u \in \mathbb{V}^g$ minimizing the functional
\begin{equation}\label{E:NMS-Energy-prescribed-nmc}
\mathcal{K}_{s}[u;f] := I_s[u] - \frac{d \alpha_d}{1-2s}\int_{\Omega} f(x) u(x) dx.
\end{equation}
Because  $I_s[u]$ is convex and the second term in the right hand side above is linear, it follows that this functional is also convex. Then, by taking the first variation of \eqref{E:NMS-Energy-prescribed-nmc}, we see that $u \in \mathbb{V}^g$ is the minimizer of $\mathcal{K}_{s}[\cdot;f]$ if and only if it satisfies
\begin{equation}\label{E:NMS-variation-prescribed-nmc}
\begin{aligned}
0 &= a_u(u,v) - \frac{d \alpha_d}{1-2s}\int_{\Omega} f(x) v(x) dx \\
&= \iint_{Q_{\Omega}} G_s\l(\frac{u(x)-u(y)}{|x-y|}\r) \frac{v(x)-v(y)}{|x-y|^{d+2s}}dx dy
- \frac{d \alpha_d}{1-2s}\int_{\Omega} f(x) v(x) dx
\end{aligned}
\end{equation}
for every $v \in \mathbb{V}^0$. Formally, \eqref{E:NMS-Prescribe-Nmc} coincides with \eqref{E:NMS-variation-prescribed-nmc} because one can multiply \eqref{E:NMS-Prescribe-Nmc} by a test function $v$, integrate by parts and take advantage of the fact that $G_s$ is an odd function to arrive at \eqref{E:NMS-variation-prescribed-nmc} up to a constant factor.

One intriguing question regarding the energy $\mathcal{K}_{s}[u;f]$ in \eqref{E:NMS-Energy-prescribed-nmc} is what conditions on $f$ are needed to guarantee that it is bounded below. In fact, for the variational formulation of the classical mean curvature problem \eqref{E:NMS-Prescribe-MC}, Giaquinta \cite{Giaq1974dirichlet-pmc} proves the following  necessary and sufficient condition for well posedness: there exists some $\ve_0 > 0$ such that for every measurable set $A \subset \Omega$,
\begin{equation}\label{E:prescribed-mc-assumption}
\Big| \int_A f(x)dx \Big| \le \frac{(1 - \ve_0)}{d} \; \mathcal{H}^{d-1}(\partial A),
\end{equation}
where $\mathcal{H}^{d-1}$ denotes the $(d-1)-$dimensional Hausdorff measure. In some sense, this condition ensures that the function $f$ be suitably small. 

Although we are not aware of such a characterization for prescribed nonlocal mean curvature problems, a related sufficient condition for $\mathcal{K}_{s}[u;f]$ to have a lower bound can be easily derived.
In fact, exploiting \cite[Lemma 2.5 and Proposition 2.7]{BoLiNo19analysis} and the Sobolev embedding $W^{2s}_1(\Omega) \subset L^{d/(d-2s)}(\Omega)$ we deduce that
\[
I_s[u] + C_1(d,\Omega,s, \|g\|_{L^\infty(\Omega^c)}) \ge |u|_{W^{2s}_1(\Omega)} \gtrsim \Vert u \Vert_{L^{d/(d-2s)}(\Omega)}.
\]
On the other hand, H\"older's inequality gives
\[
\int_{\Omega} f(x)u(x) dx \le \Vert u \Vert_{L^{d/(d-2s)}(\Omega)} \|f\|_{L^{d/(2s)}(\Omega)},
\]
whence $\mathcal{K}_{s}[u;f]$ is bounded from below provided $\Vert f \Vert_{L^{d/(2s)}(\Omega)}$ is suitably small,
\[
\mathcal{K}_{s}[u;f] \ge \Vert u \Vert_{L^{d/(d-2s)}(\Omega)} \l(C -   \|f\|_{L^{d/(2s)}(\Omega)}\r)  -  C_1(d,\Omega,s, \|g\|_{L^\infty(\Omega^c)}).
\]
This is to some extent consistent with \eqref{E:prescribed-mc-assumption}, because it holds that
\[
\Big| \int_A f(x)dx \Big| \le \l( \int_A 1 dx \r)^{\frac{d-1}{d}} \l( \int_A |f(x)|^d dx \r)^{\frac{1}{d}} \lesssim \mathcal{H}^{d-1}(\partial A) \Vert f \Vert_{L^d(\Omega)},
\]
due to H\"older's inequality and the isoperimetric inequality, and formally the case $2s = 1$ corresponds to the classical prescribed mean curvature problem (cf. \eqref{E:NMS-NonLocalCurv-asymp}).

\section{Numerical experiments} \label{sec:experiments}
This section presents a variety of numerical experiments that illustrate some of the main features of fractional minimal graphs discussed in this paper. From a quantitative perspective, we explore stickiness and the effect of truncating the computational domain. Moreover, we report on the conditioning of the matrices arising in the iterative resolution of the nonlinear discrete equations. Our experiments also illustrate that nonlocal minimal graphs may change their concavity inside the domain $\Omega$, and we show that graphs with prescribed fractional mean curvature may be discontinuous in $\Omega$.

In all the experiments displayed in this section we use the damped Newton algorithm from \S\ref{sec:Newton}. We refer to \cite{BoLiNo19analysis} for experiments involving the semi-implicit gradient flow algorithm and illustrating its energy-decrease property.

\subsection{Quantitative boundary behavior} \label{sec:stickiness_thm2}

We first consider the example studied in \cite[Theorem 1.2]{DipiSavinVald17}. We solve \eqref{E:WeakForm-discrete} for $\Omega = (-1,1) \subset \mR$ and $g(x) = M\textrm{sign}(x)$, where $M > 0$. Reference \cite{DipiSavinVald17} proves that, for every $s \in (0,1/2)$, stickiness (i.e. the solution being discontinuous at $\partial \Omega$) occurs if $M$ is big enough and, denoting the corresponding solution by $u^M$, that there exists an optimal constant $c_0$ such that
\begin{equation} \label{eq:bdry-behavior-1d}
\sup_{x \in \Omega} u^M(x) < c_0 M^{\frac{1+2s}{2+2s}}, \quad
\inf_{x \in \Omega} u^M(x) > -c_0 M^{\frac{1+2s}{2+2s}}.
\end{equation}
In our experiments, we consider $s=0.1,0.25,0.4$ and use graded meshes (cf. \Cref{sec:graded}) with parameter $\mu = 2, h = 10^{-3}$ to better resolve the boundary discontinuity. The mesh size $h$ here is taken in such a way that the resulting mesh partitions $\Omega = (-1,1)$ into $\lfloor \frac{|\Omega|^{1/\mu}}{h}\rfloor$ subintervals and the smallest ones have size $h^\mu$. Moreover, since this is an example in one dimension and the unboundedly supported data $g$ is piecewise constant, we can use quadrature to approximate the integrals over $\Omega^c$ rather than directly truncating $g$. The left panel in \Cref{fig:stickiness_thm2} shows the computed solutions with $M = 16$.

In all cases we observe that the discrete solutions $u_h$ are monotonically increasing in $\Omega$, so we let $x_1$ be the free node closest to $1$ and use $u^M_h(x_1)$ as an approximation of $\sup_{x \in \Omega} u^M(x)$. The right panel in \Cref{fig:stickiness_thm2} shows how $u^M_h(x_1)$ varies with respect to $M$ for different values of $s$. 

\begin{figure}[!htb]
	\begin{center}
		\includegraphics[width=0.45\linewidth]{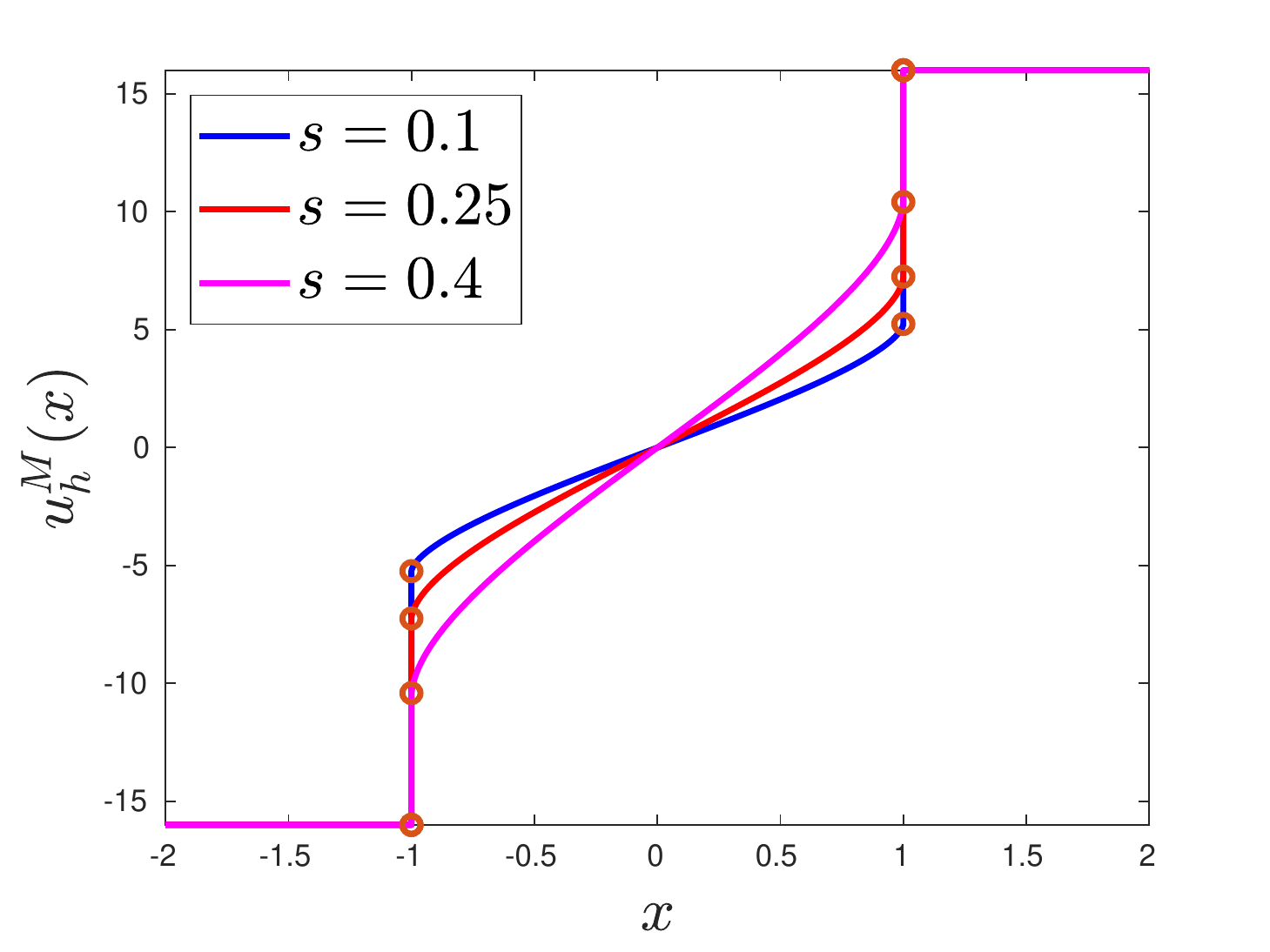} \hspace{-0.7cm}
		\includegraphics[width=0.45\linewidth]{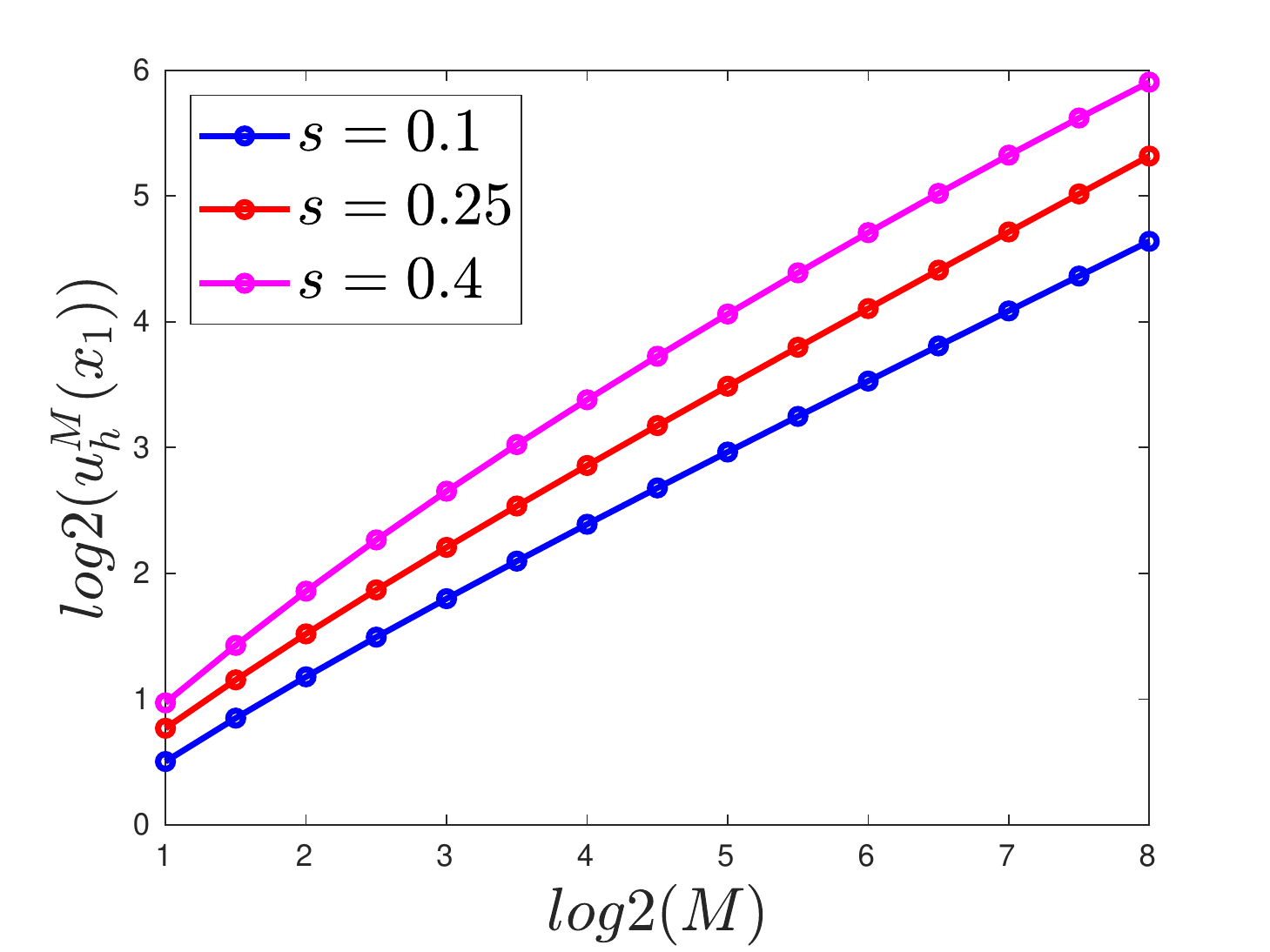} 
		\caption{
Stickiness in 1d. In the setting of \Cref{sec:stickiness_thm2}, the left panel displays the finite element solutions $u^M_h$ for $M = 16$ computed over graded meshes with parameters $\mu = 2, h = 10^{-3}$ and $s\in \{0,1, 0.25, 0.4\}$. The right panel shows the value of $u^M_h(x_1)$ as a function of $M$ for $s \in \{0,1, 0.25, 0.4\}$, which is expected to behave according to \eqref{eq:bdry-behavior-1d}.
}\label{fig:stickiness_thm2}
	\end{center}
\end{figure}

For $s=0.1$ and $s=0.25$ the slopes of the curves are slightly larger than the theoretical rate $M^{\frac{1+2s}{2+2s}}$ whenever $M$ is small. However, as $M$ increases, we see a good agreement with theory. Comparing results for $M=2^{7.5}$ and $M=2^8$, we observe approximate rates $0.553$ for $s=0.1$ and $0.602$ for $s=0.25$, where the expected rates are $6/11 \approx 0.545$ and $3/5 = 0.600$, respectively. However, the situation is different for $s=0.4$: the plotted curve does not correspond to a flat line, and the last two nodes plotted, with $M=2^{7.5}$ and $M=2^8$, show a relative slope of about $0.57$, which is off the expected $9/14 \approx 0.643$.

We believe this issue is due to the mesh size $h$ not being small enough to resolve the boundary behavior. We run the same experiment on a finer mesh, namely with $h = 10^{-4}, \mu = 2$, and report our findings for $s=0.4$ and compare them with the ones for the coarser mesh on \Cref{tab:stickiness_s04}. The results are closer to the predicted rate.

\begin{table}
	\begin{center}
	\begin{tabular}{|c|c|c|c|c|}
	\hline
 &	\multicolumn{2}{|c|}{Example with $h = 10^{-3}$} &	\multicolumn{2}{|c|}{Example with $h = 10^{-4}$} \\
	\hline
	$\log_2(M)$ & $u_h^M(x_1)$ & Slope  & $u_h^M(x_1)$ & Slope \\
	\hline
	$6.0$ & $26.1545$ & \texttt{N/A}   & $26.7488$ & \texttt{N/A}   \\ 
	\hline
	$6.5$ & $32.4687$ & $0.624$ & $33.4057$ & $0.641$\\ 
	\hline
	$7.0$ & $40.0845$ & $0.608$ & $41.5497$ & $0.629$ \\ 
	\hline
	$7.5$ & $49.1873$  & $0.590$ & $51.4627$ & $0.617$  \\
	\hline
	$8.0$ & $59.9410$ & $0.571$  & $63.4528$ & $0.604$ \\
	\hline	
	\end{tabular}
        \bigskip
	\caption{Comparison between computational results for the problem described in \Cref{sec:stickiness_thm2} over two different meshes for $s=0.4$. 
Let $M_i$ be the value of $M$ in the $i$-th row. In this table, by the slope at $M_i$ we refer to $\frac{\log(u_h^{M_i}(x_1)) -\log(u_h^{M_{i-1}}(x_1))}{\log(M_i) - \log(M_{i-1})}$ that, according to \eqref{eq:bdry-behavior-1d}, is expected to be equal to $9/14 \approx 0.643$.
} \label{tab:stickiness_s04}
\end{center}
\end{table}

\subsection{Conditioning} \label{sec:conditioning}

For the solutions of the linear systems arising in our discrete formulations, we use a conjugate gradient method. Therefore, the number of iterations needed for a fixed tolerance scales like $\sqrt{\kappa(K)}$, where $\kappa(K)$ is the condition number of the stiffness matrix $K$. For linear problems of order $s$ involving the fractional Laplacian $(-\Delta)^s$, the condition number of $K$ satisfies
\cite{AiMcTr99}
\[
\kappa(K) = \mathcal{O}\left( N^{2s/d} \left( \frac{h_{max}}{h_{min}} \right)^{d-2s} \right). 
\]
Reference \cite{AiMcTr99} also shows that diagonal preconditioning yields $\kappa(K) = \mathcal{O}\left( N^{2s/d} \right)$, where $N$ is the dimension of the finite element space.

Using the Matlab function \verb+condest+, we estimate the condition number of the Jacobian matrix in the last Newton iteration in the example from \Cref{sec:stickiness_thm2} with $M=1$, with and without diagonal preconditioning. \Cref{fig:condition_number_thm2_s01_s04} summarizes our findings.

\begin{figure}[!htb]
	\begin{center} 
		\includegraphics[width=0.45\linewidth]{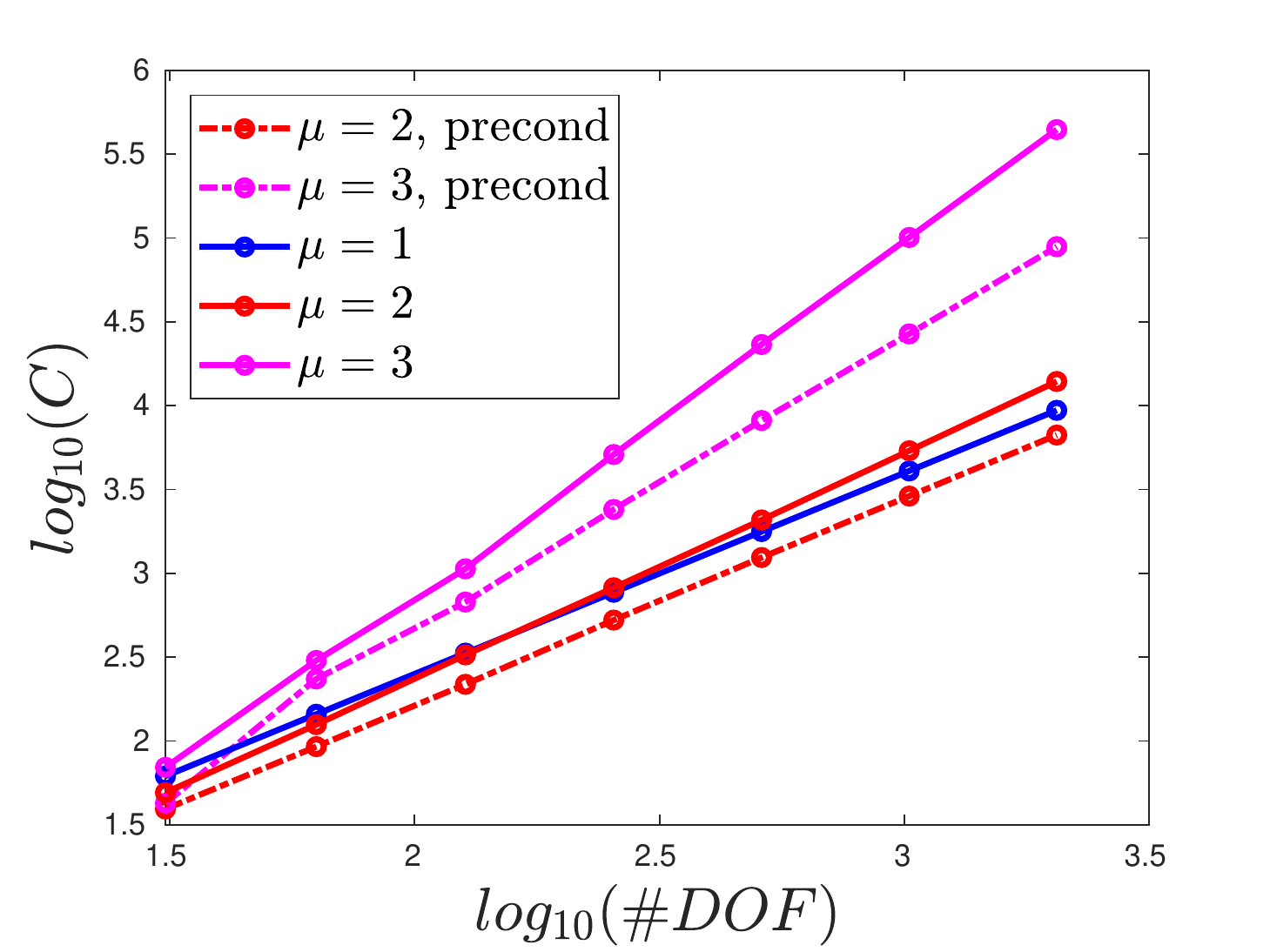} \hspace{-0.7cm}
		\includegraphics[width=0.45\linewidth]{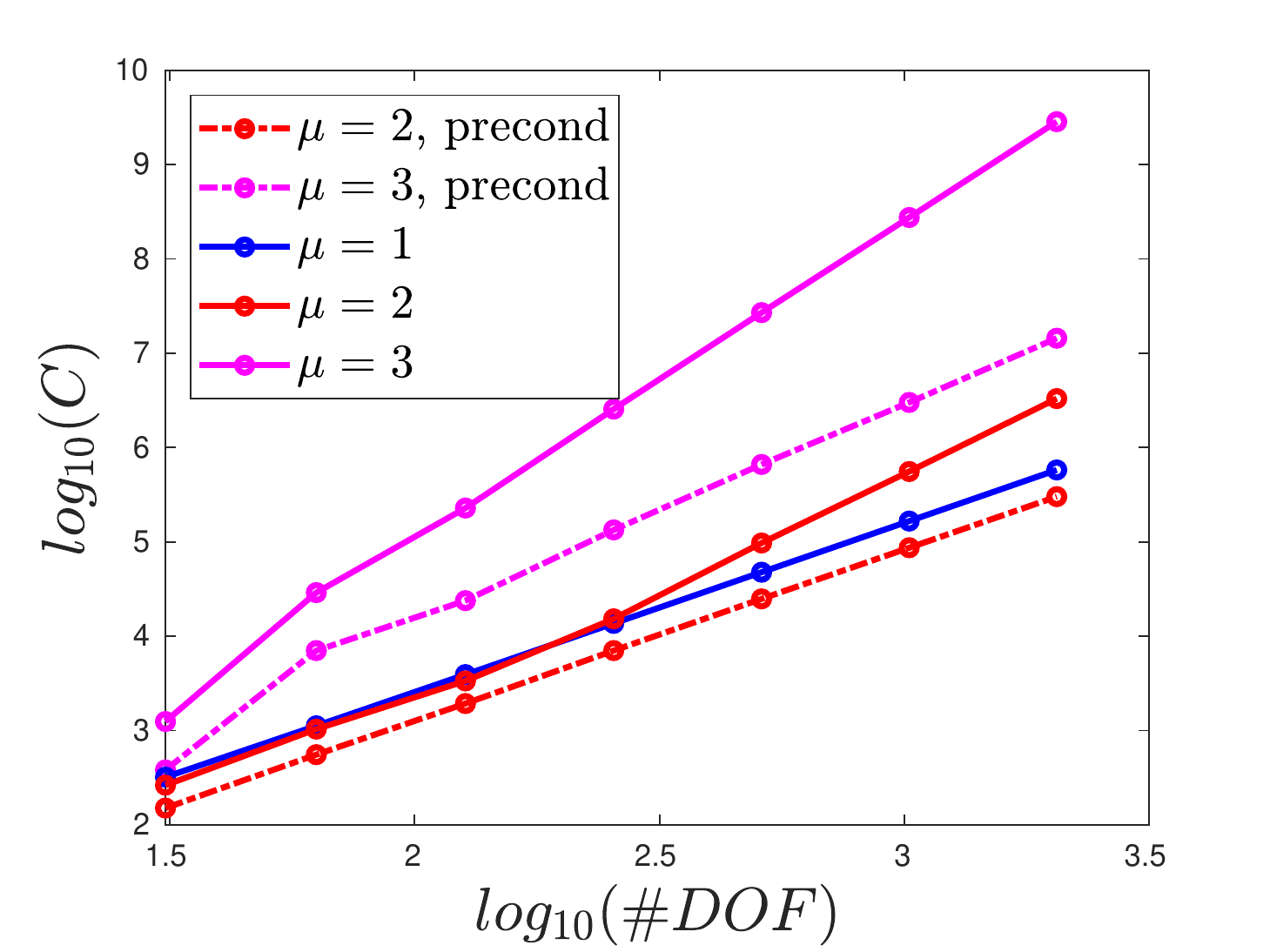}		
	\end{center}
	\caption{Condition numbers of the Jacobian matrix $\widetilde{K}$ at the last step of \Cref{alg:Damped-Newton} for the problem described in \Cref{sec:stickiness_thm2} with $s = 0.1$ (left), $s = 0.4$ (right) and meshes with grading parameters $\mu \in \{ 1,2, 3\}$. The condition number on quasi-uniform meshes ($\mu = 1$) scales as $N^{2s+1}$, in agreement with the $s$-fractional mean curvature operator being an operator of order $s+1/2$ (cf. \eqref{E:WeakForm-NMS-Graph}). While the conditioning for graded meshes is significantly poorer, when $\mu = 2$ diagonal preconditioning recovers condition numbers comparable to the ones on quasi-uniform meshes.
}
\label{fig:condition_number_thm2_s01_s04}
\end{figure}

Let $N = \mbox{dim} \mathbb{V}_h^g$ be the number of degrees of freedom. For a fixed $s$ and using uniform meshes, we observe that the condition number behaves like $N^{2s+1} \simeq h^{-2s-1}$: this is consistent with the $s$-fractional mean curvature operator being an operator of order $s+1/2$. For graded meshes (with $\mu = 2, \mu=3$), the behavior is less clear. When using diagonal preconditioning for $\mu = 2$, we observe that the condition number also behaves like $N^{2s+1}$.

\subsection{Truncation of unboundedly supported data}
In \Cref{sec:support_g}, we studied the effect of truncating unboundedly supported data and proved the convergence of the discrete solutions of the truncated problems $u^H_h$ towards $u$ as $h \to 0$, $H\to \infty$. 

Here, we study numerically the effect of data truncation by running experiments on a simple two-dimensional problem. Consider $\Omega = B_1 \subset \mR^2$ and $g \equiv 1$; then, the nonlocal minimal graph $u$ is a constant function. For $H > 0$, we set $\Omega_H = B_{H+1}$. and compute nonlocal minimal graphs on $\Omega$ with Dirichlet data $g^H = \chi_{\Omega_H}$, which is a truncation of $g \equiv 1$. Clearly, if there was no truncation, then $u_h$ should be constantly $1$; the effect of the truncation of $g$ is that the minimum value of $u^H_h$ inside $\Omega$ is strictly less than $1$. For $s = 0.25$, we plot the $L^1(\Omega)$ and $L^{\infty}(\Omega)$ norms of $u_h - u^H_h$ as a function of $H$ in \Cref{fig:truncation_2d_s025}. The slope of the curve is close to $-1.5$ for large $H$, which is in agreement with the $\mathcal{O}(H^{-1-2s})$ consistency error for the energy $I_s$ we proved in \Cref{cor:consistency_truncation}.

\begin{figure}[!htb]
	\begin{center}
		\includegraphics[width=0.5\linewidth]{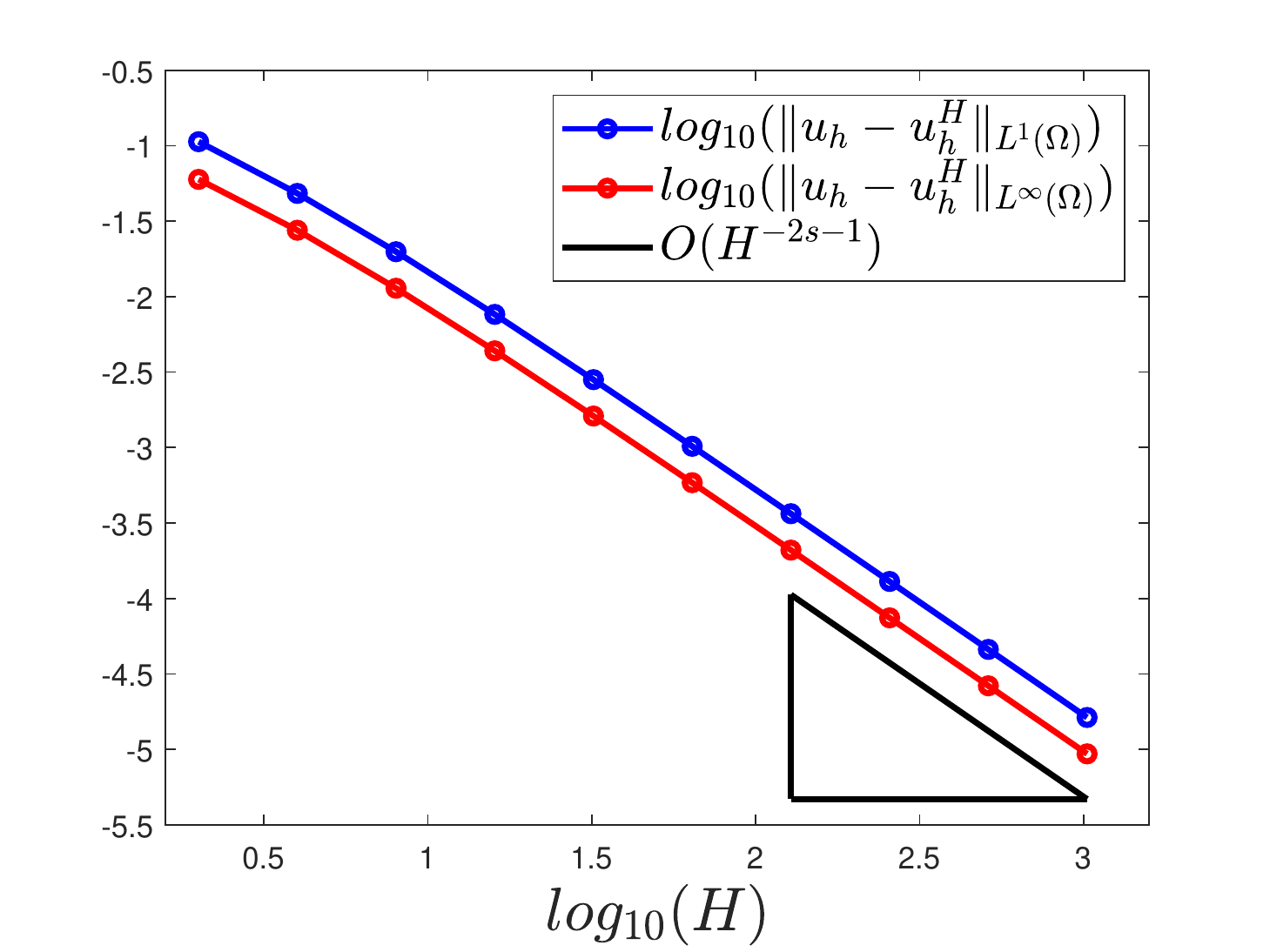}
	\end{center}
	\caption{Effects of truncation in $2d$ for $s=0.25$: for $g^H = \chi_{\Omega_H}$, we compute the $L^1$ and $L^\infty$ discrepancies between $u_h\equiv 1$ and $u^H_h$ as a function of $H$. For both norms we observe a discrepancy of order $H^{-1-2s}$, in agreement with \Cref{cor:consistency_truncation}.
}
	\label{fig:truncation_2d_s025}
\end{figure}

\subsection{Change of convexity} \label{sec:convexity}
This is a peculiar behavior of fractional minimal graphs. 
We consider $\Omega = (-1,1)$, $s = 0.02$, $g(x) = 1$ for $a \le |x| \le 2$ and $g(x) = 0$ otherwise, and denote by $u_a$ the solution of \eqref{E:WeakForm-discrete}. For $a = 1$, it is apparent from \Cref{fig:change_convexity_1d_details} (left panel) that the solution $u_1$ is convex in $\Omega$ and has stickiness on the boundary. In addition, the figure confirms that $\lim\limits_{x \to 1^{-}} u_a'(x) = \infty$, which is asserted in \cite[Corollary 1.3]{DipiSavinVald19nonlocal}. On the contrary, for $1 < a < 2$, as can be seen from \Cref{fig:change_convexity_1d_details} (right panel), \cite[Corollary 1.3]{DipiSavinVald19nonlocal} implies that $\lim\limits_{x \to 1^{-}} u_a'(x) = -\infty$ since $g(x) = 0$ near the boundary of $\Omega$. 
This fact implies that $u(x)$ cannot be convex near $x = 1$ for $1 < a < 2$. Furthermore, as $a \to 1^+$ one expects that $u_a(x) \to u_1(x)$ and thus that $u_a$ be convex in the interior of $\Omega = (-1,1)$ for $a$ close to $1$. Therefore it is natural that for some values of $a > 1$ sufficiently close to $1$, the solution $u_a$ changes the sign of its second derivative inside $\Omega$. In fact, we see from the right panel in \Cref{fig:change_convexity_1d_details} that the nonlocal minimal graph $u$ in $\Omega$ continuously changes from a convex curve into a concave one as $a$ varies from $1$ to $1.5$.

\begin{figure}[!htb]
	\begin{center}
		\includegraphics[width=0.45\linewidth]{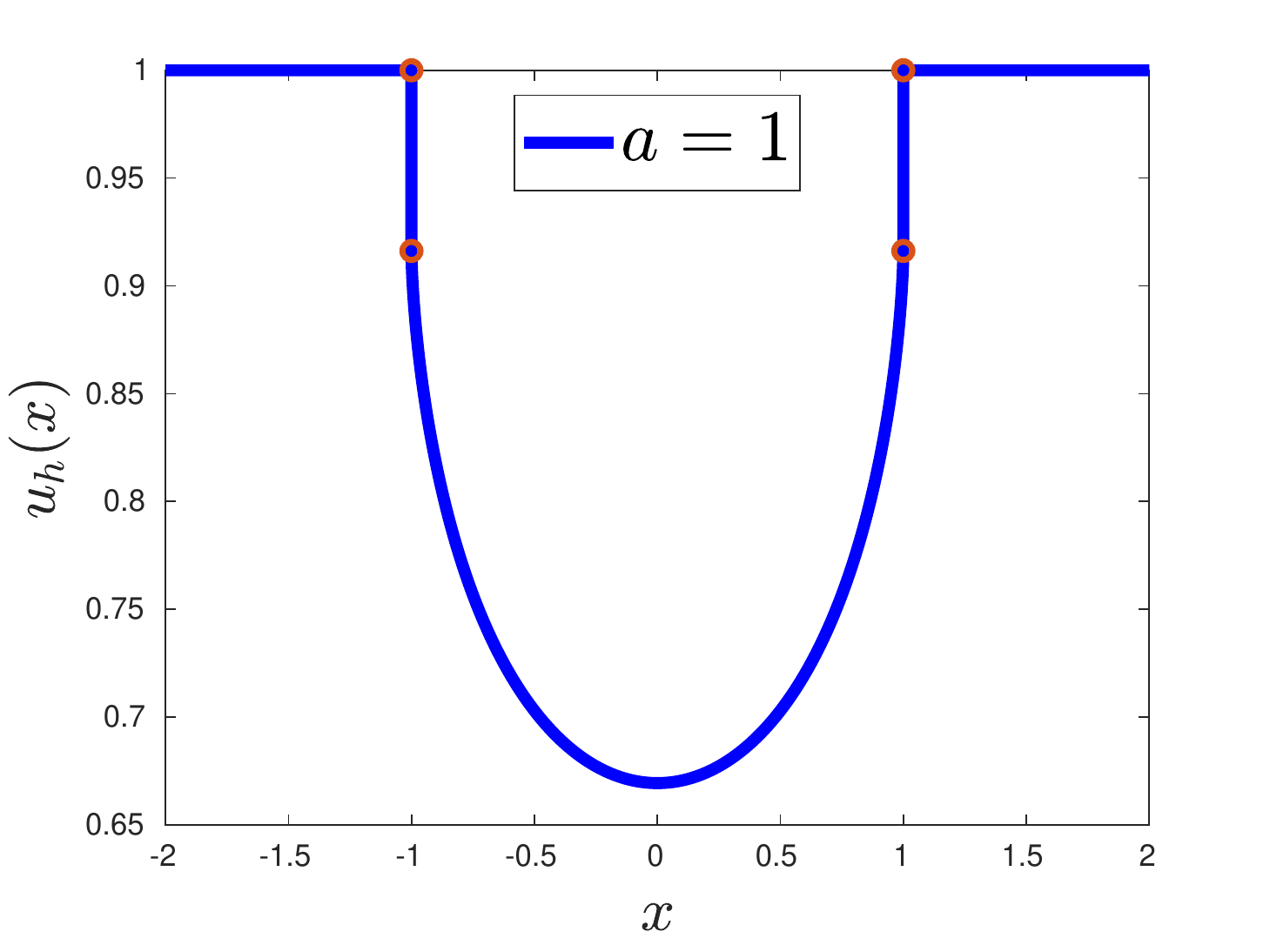}
		\hspace{-0.7cm}
		\includegraphics[width=0.45\linewidth]{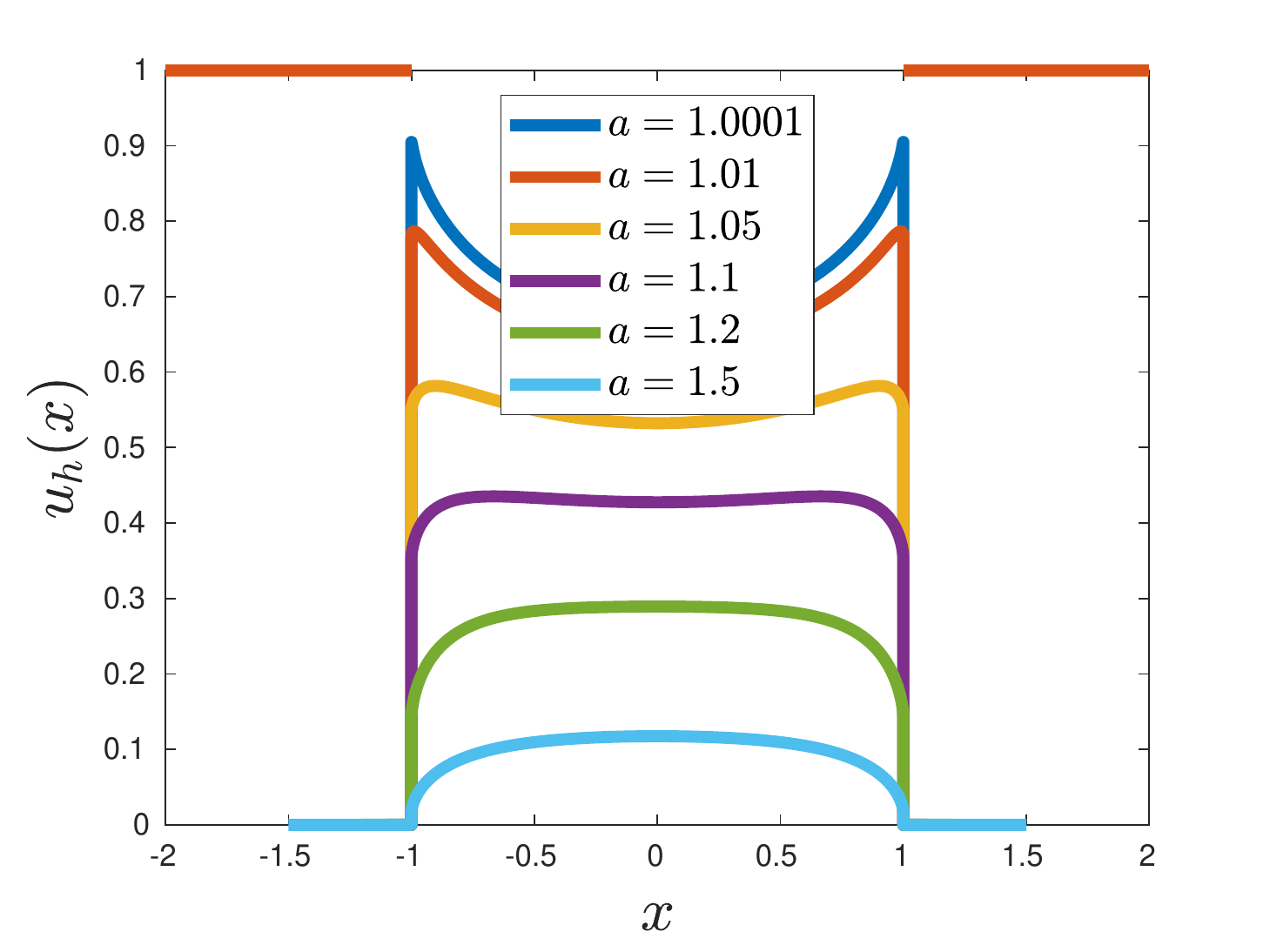}
	\end{center}
	\caption{Change of convexity: one-dimensional experiment for $s = 0.02$ with $a = 1$ (left panel) and $a = 1.0001, 1.01, 1.05,$ $1.1, 1.2, 1.5$ (right panel). The solutions $u_a$ exhibit a transition from being convex in $\Omega$ for $a=1$ to being concave for $a = 1.5$.}
	\label{fig:change_convexity_1d_details}
\end{figure}

This change of convexity is not restricted to one-dimensional problems. Let $\Omega \subset \mR^2$ be the unit ball, $s = 0.25$, and $g(x) = 1$ for $\frac{129}{128} \le |x| \le 1.5$ and $g(x) = 0$ otherwise. \Cref{fig:change_convexity} (right panel) shows a radial slice of the discrete minimal graph, which is a convex function near the origin but concave near $\pp\Omega$. An argument analogous to the one we discussed in the previous paragraph also explains this behavior in a two-dimensional experiment.

\begin{figure}[!htb]
	\begin{center}
		\includegraphics[width=0.45\linewidth]{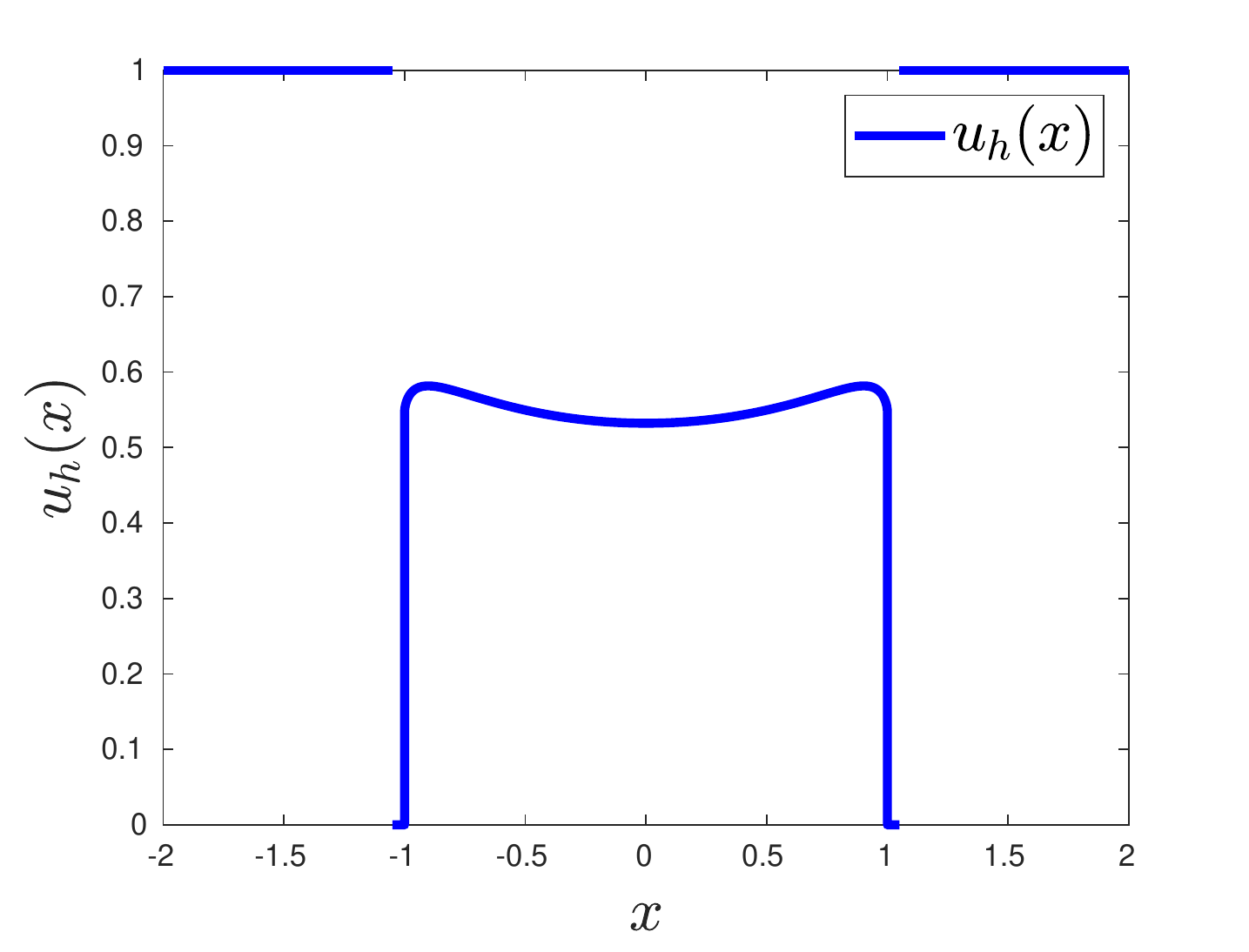}
		\hspace{-0.7cm}
		\includegraphics[width=0.45\linewidth]{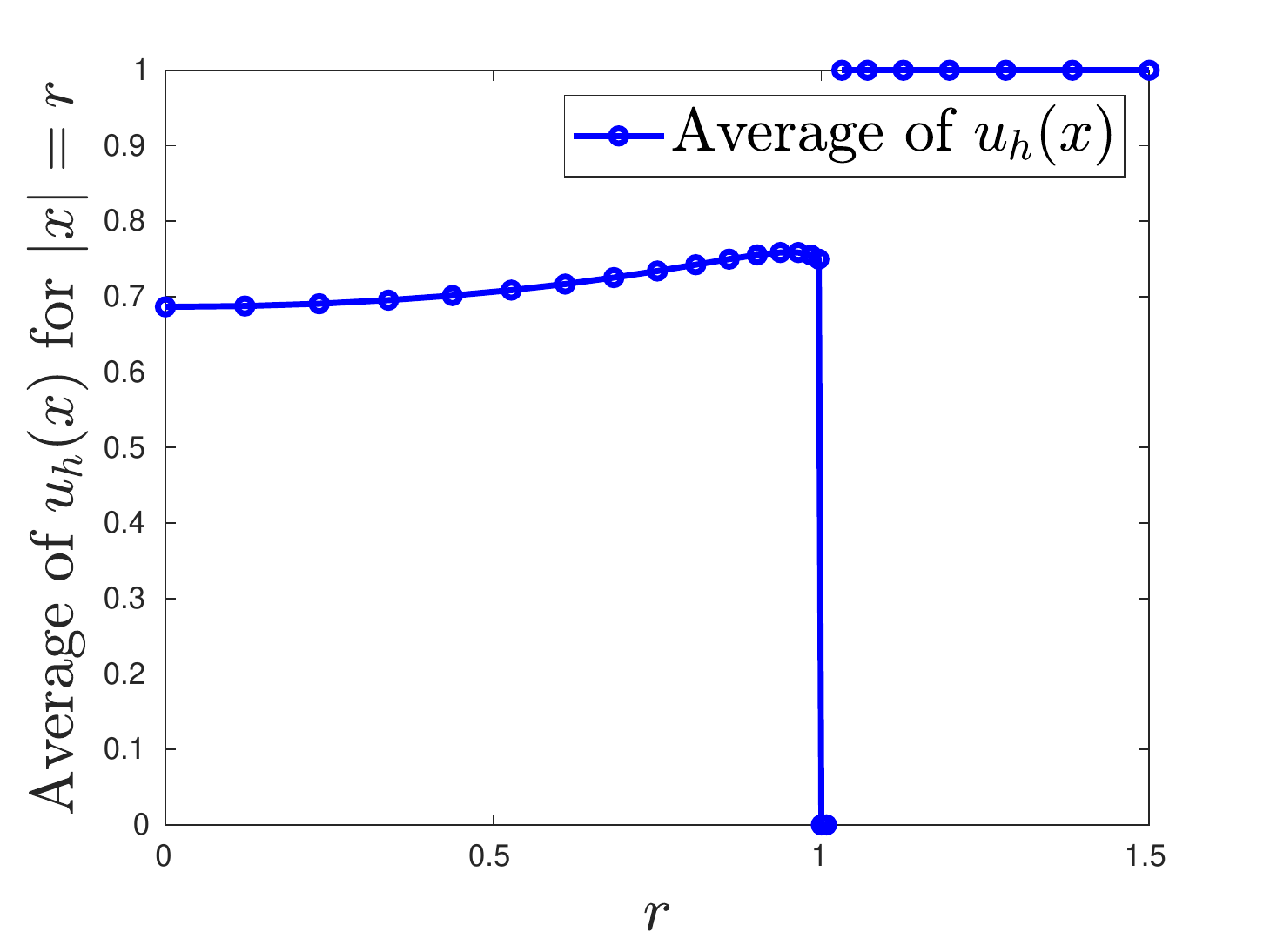}
	\end{center}
	\caption{Change of convexity: one-dimensional experiment with $s=0.02$ (left panel) and two-dimensional experiment with $s=0.25$ (right panel). The piecewise constant boundary data vanish near the boundary of $\Omega$ and at infinity and are equal to $1$ on an intermediate annulus.
	} \label{fig:change_convexity}
\end{figure}

\subsection{Geometric rigidity}
\label{sec:behavior_2d}
Stickiness is one of the intrinsic and distintive features of nonlocal minimal graphs. It can be delicate especially in dimension more than one. We now analyze a problem studied in \cite{DipiSavinVald19boundary} that illustrates the fact that for $\Omega \subset \mR^2$, if nonlocal minimal graphs are continuous at some point $x \in \pp\Omega$ then they must also have continuous tangential derivatives at such a point. This geometric rigidity stands in sharp contrast with the case of either fractional-order linear problems and classical minimal graphs.

Specifically, we consider $\Omega = (0,1) \times (-1,1)$ and the Dirichlet data
\[
g(x, y) = \gamma \l( \chi_{(-1,-a) \times (0,1)}(x,y) - \chi_{(-1,-a) \times (-1,0)}(x,y) \r)
\]
where $a \in [0,1]$ and $\gamma > 0$ are parameters to be chosen. 
We construct graded meshes with $\mu = 2$ and smallest mesh size $h^{\mu} = 2^{-7}$; see Section \ref{sec:graded}. \Cref{fig:2d-stick1-uh} (left panel) displays the numerical solution $u_h$ associated with $s=0.25,$ $\gamma = 2$ and $a = 1/8$.  

If one defines the function $u_0(y) = \lim_{x \to 0^+} u(x, y)$, then according to \cite[Theorem 1.4]{DipiSavinVald19boundary}, one has $u'_0(0) = 0$ for $a > 0$. We run a sequence of experiments to computationally verify this theoretical result. For meshes with $\mu = 2$ and $h^{\mu} = 2^{-7}, 2^{-8}, 2^{-9}$, the slopes of $u_h$ in the $y$-direction at $(x,0)$ for $x = 2^{-6},2^{-7}, 2^{-8},2^{-9}$, are recorded in  \Cref{tab:slope_stickiness} below for $s = 0.1, 0.25, 0.4$. Because computing the slope of $u_h$ at $(x,0)$ would be meaningless when $x$ is smaller than $h^{\mu}$, we write a \texttt{N/A} symbol in those cases. Our experiments show that the slopes decrease as $x$ approaches $0$.

To further illustrate this behavior, in \Cref{fig:2d-stick1-uh} (right panel) we display the computed solutions $u_h(x,y)$ at $x=2^{-3}, 2^{-6}, 2^{-9}$, for $s=0.25$ over a mesh with $h^{\mu} = 2^{-9}$. The flattening of the curves as $x \to 0^+$ is apparent. 

\begin{figure}[!htb]
\includegraphics[width=0.4\linewidth]{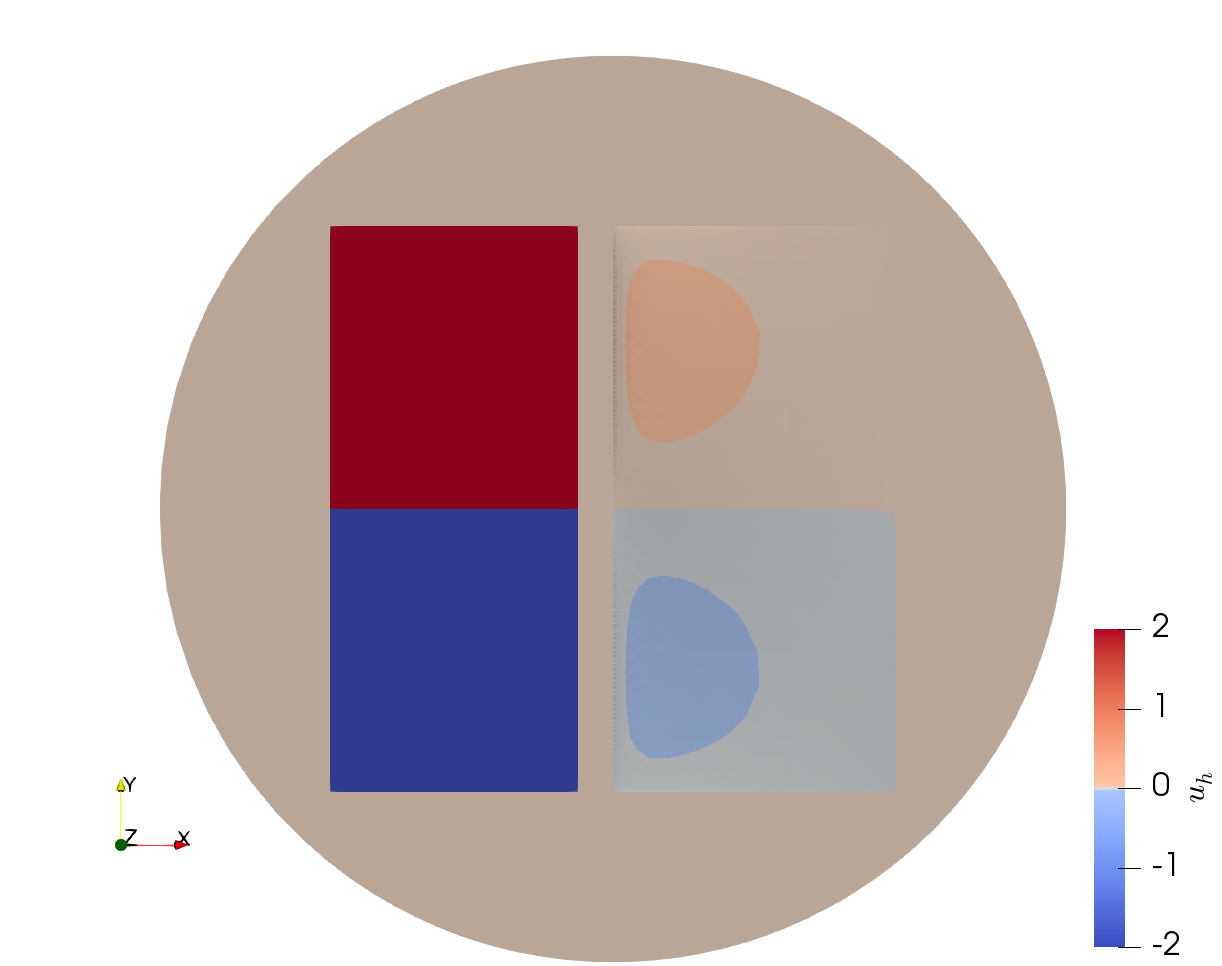}
\includegraphics[width=0.5\linewidth]{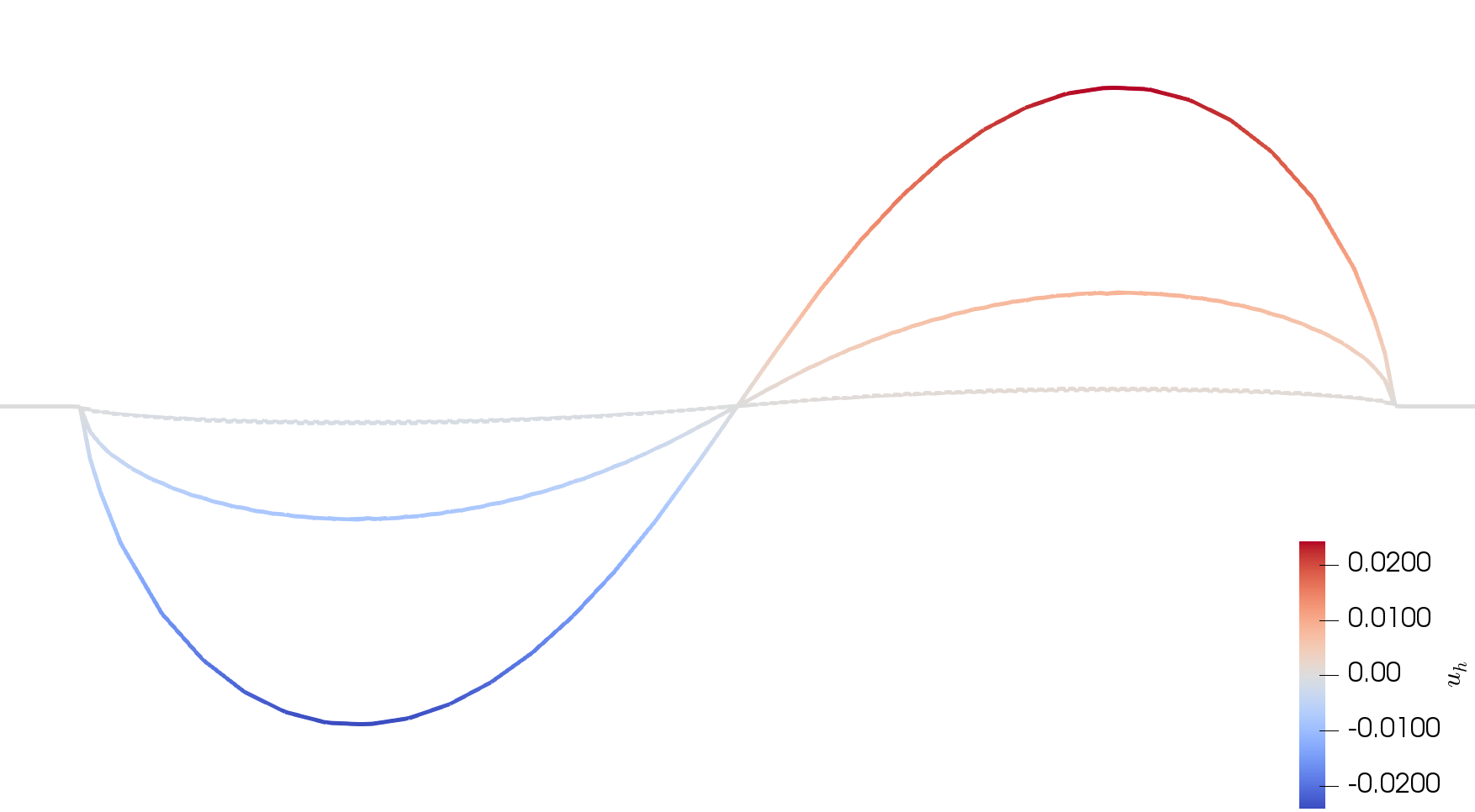}
	\caption{Plot of $u_h$ in \Cref{sec:behavior_2d} for $\gamma = 2, a = 1/8$ and $s=0.25$. Left panel: top view of the solution. Right panel: slices at $x=2^{-3}, 2^{-6}$ and $2^{-9}$. The fractional minimal graph flattens as $x \to 0^+$, in agreement with the fact that for such a minimizer being continuous at some point $x\in \pp\Omega$ implies having continuous tangential derivatives at such a point.
}
	\label{fig:2d-stick1-uh}
\end{figure}

\begin{table}
	\begin{center}
\begin{tabular}{|c|c|c|c|c|} \hline
	\multicolumn{5}{|c|}{$s = 0.10$} \\
	\hline	
	$h^{\mu}$  & $x = 2^{-9}$ & $x = 2^{-8}$ & $x = 2^{-7}$ & $x = 2^{-6}$ \\ \hline	
	$2^{-7}$ & \texttt{N/A}                     &\texttt{N/A}                     & $8.546\times10^{-2}$ & $1.1945\times10^{-1}$  \\	
	$2^{-8}$ & \texttt{N/A}                    & $5.856\times10^{-2}$ & $8.406\times10^{-2}$ & $1.2140\times10^{-1}$ \\	
	$2^{-9}$ & $3.940\times10^{-2}$ & $5.730\times10^{-2}$ & $8.572\times10^{-2}$ & $1.2332\times10^{-1}$  \\	
	\hline  
\multicolumn{5}{c}{ \ }	\\
\hline	
	\multicolumn{5}{|c|}{$s = 0.25$} \\
	\hline		
	$h^{\mu}$  & $x = 2^{-9}$ & $x = 2^{-8}$ & $x = 2^{-7}$ & $x = 2^{-6}$ \\ \hline	
	$2^{-7}$ & \texttt{N/A}                    & \texttt{N/A}                       & $3.466\times10^{-2}$ & $5.473\times10^{-2}$  \\	
	$2^{-8}$ &\texttt{N/A}                     & $2.135\times10^{-2}$ & $3.469\times10^{-2}$ & $5.551\times10^{-2}$ \\	
	$2^{-9}$ & $1.289\times10^{-2}$ & $2.126\times10^{-2}$ & $3.543\times10^{-2}$ & $5.640\times10^{-2}$  \\	
	\hline  	
%
\multicolumn{5}{c}{ \ }	\\

\hline
	\multicolumn{5}{|c|}{$s = 0.40$} \\
	\hline		
	$h^{\mu}$  & $x = 2^{-9}$ & $x = 2^{-8}$ & $x = 2^{-7}$ & $x = 2^{-6}$ \\ \hline
	$2^{-7}$ & \texttt{N/A}                    & \texttt{N/A}                     & $8.605\times10^{-3}$ & $1.509\times10^{-2}$  \\
	$2^{-8}$ & \texttt{N/A}                     & $4.763\times10^{-3}$ & $8.613\times10^{-3}$ & $1.540\times10^{-2}$ \\
	$2^{-9}$ & $2.578\times10^{-3}$ & $4.739\times10^{-3}$ & $8.886\times10^{-3}$ & $1.574\times10^{-2}$  \\
	\hline  
\end{tabular}
\bigskip
\caption{Example of \Cref{sec:behavior_2d}: experimental slopes $\partial_y u_h(x,0)$ for $x=2^{-k}$ and $k=6,\ldots,9$. As $x\to 0^+$, these slopes become smaller; this geometric rigidity is easier to capture for larger $s$.
}
\label{tab:slope_stickiness}
\end{center}
\end{table}

\subsection{Prescribed nonlocal mean curvature}
This section presents experiments involving graphs with nonzero prescribed mean curvature. We run experiments that indicate the need of a compatibility condition such as \eqref{E:prescribed-mc-assumption}, the fact that solutions may develop discontinuities in the interior of the domain, and point to the relation between stickiness and the nonlocal mean curvature of the domain.

\subsubsection{Compatibility}
As discussed in \Cref{sec:prescribed_curvature}, the prescribed nonlocal mean curvature problem \eqref{E:NMS-variation-prescribed-nmc} may not have solutions for some functions $f$. To verify this, in \Cref{fig:compatibility-prescribed-nmc} we consider $\Omega = B(0,1) \subset \mR^2$, $s=0.25$, $g = 0$ and two choices of $f$. For the picture on the right ($f=-10$), the residue does not converge to $0$, and the energy $\mathcal{K}_{s}[u;f]$ goes from $0$ initially down to $-6.6 \times 10^6$ after $16$ Newton iterations.

\begin{figure}[!htb]
	\begin{center}
		\includegraphics[width=0.45\linewidth]{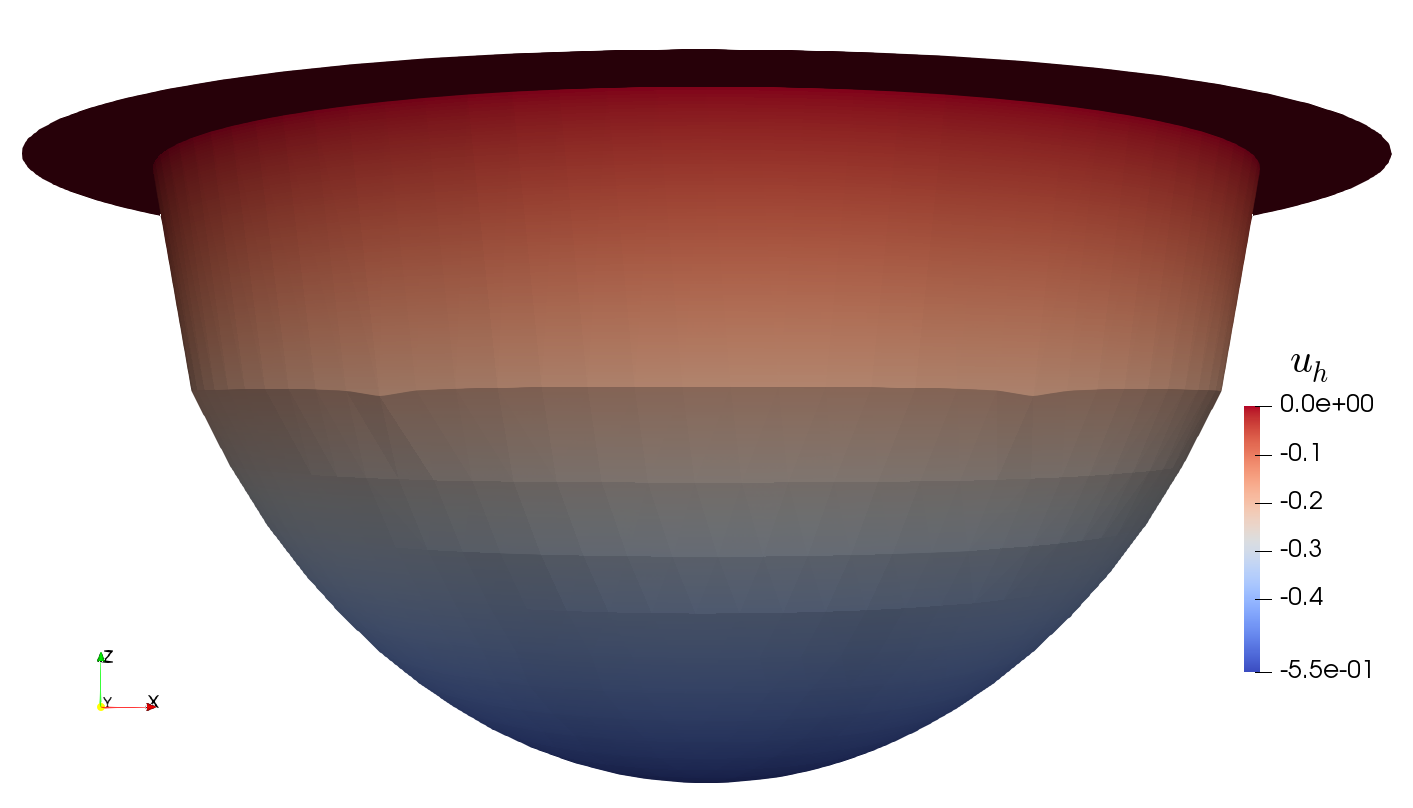}
		\includegraphics[width=0.45\linewidth]{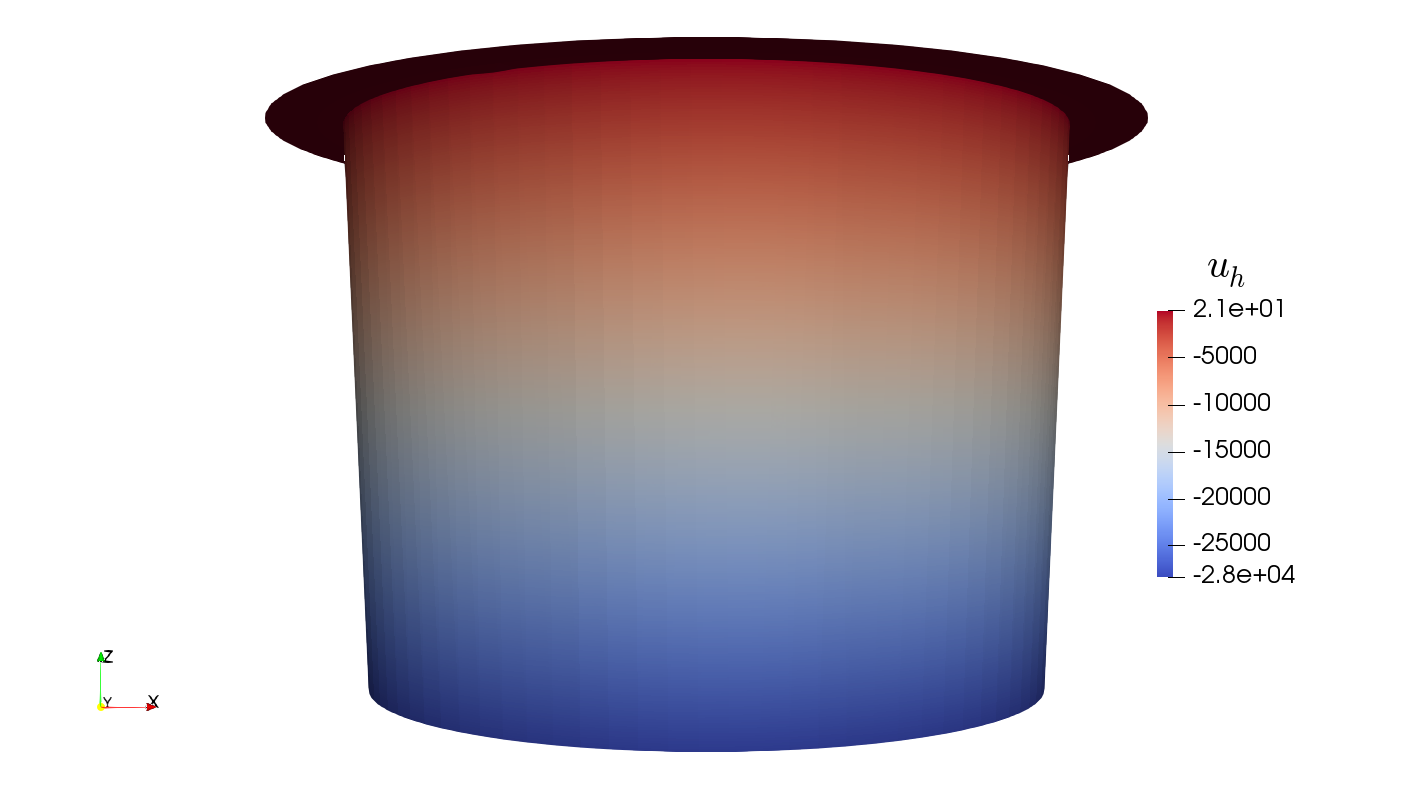}
	\end{center}
	\caption{Compatibility of data: plots of $u_h$ for $s = 0.25, f = -1$ in $\Omega$ (left) and after $16$ Newton iterations for $f = -10$ in $\Omega$ (right). The right hand side $f=-10$ turns out to be incompatible for the prescribed nonlocal mean curvature problem in $\Omega = B(0,1)$.} \label{fig:compatibility-prescribed-nmc}
\end{figure}

\subsubsection{Discontinuities}
Another interesting phenomenon we observe is that, for a discontinuous $f$, the solution $u$ may also develop discontinuities inside $\Omega$. We present the following two examples for $d=1$ and $d=2$.

In first place, let $\Omega = (-1,1) \subset \mR$, $s=0.01$, $g = 0$ and consider $f(x) = 1.5 \, \textrm{sign}(x)$. We use a mesh graded toward $x=0, \, \pm 1$ with $N = 2000$ degrees of freedom and plot the numerical solution $u_h$ in \Cref{fig:1d_discontinuity_inside}. The behavior of $u_h$ indicates that the solution $u$ has discontinuities both at $x = \pm 1$ and $x=0$.
\begin{figure}[!htb]
	\begin{center}
		\includegraphics[width=0.45\linewidth]{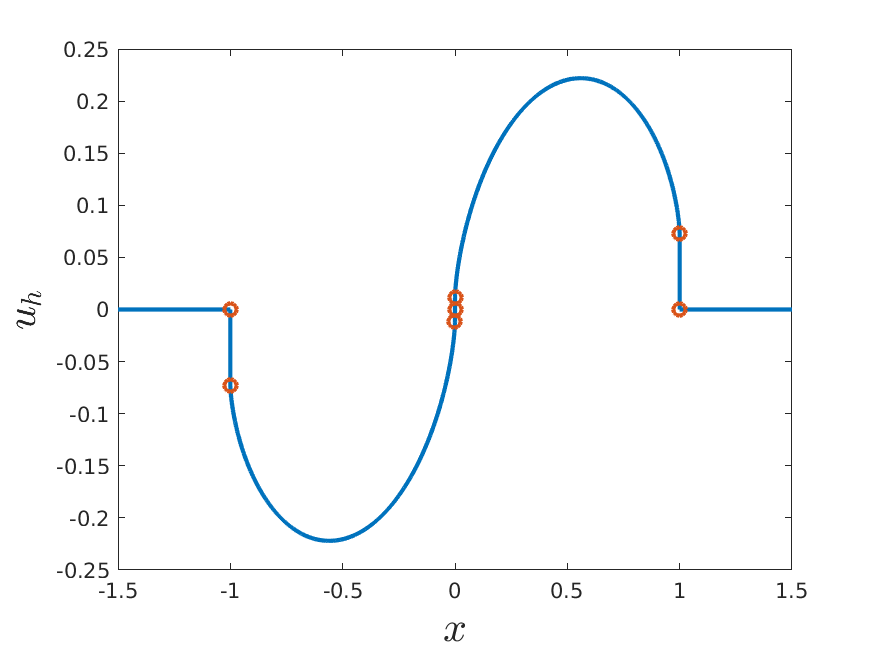}
		\includegraphics[width=0.45\linewidth]{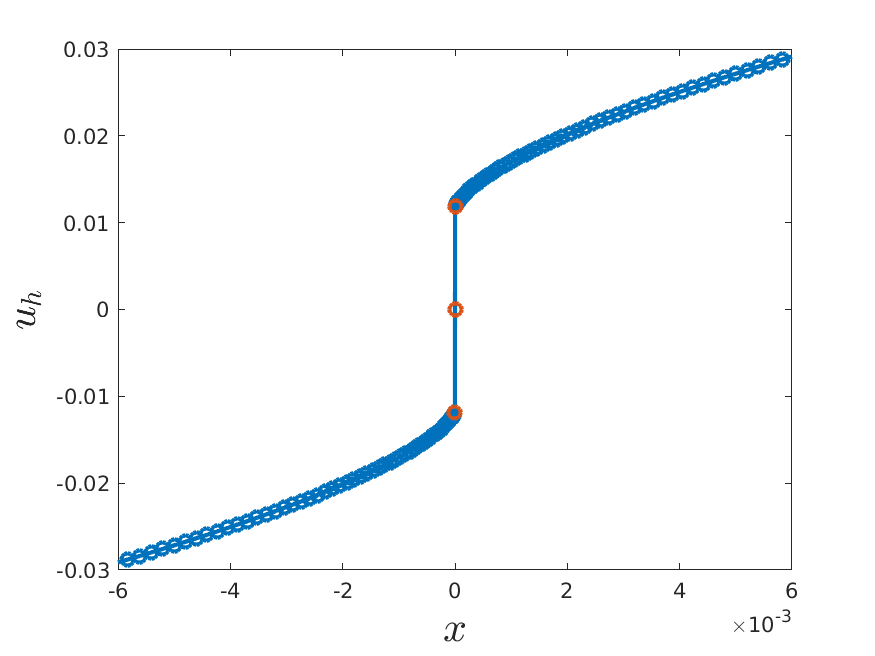}
	\end{center}
	\caption{Nonlocal minimal graph with prescribed discontinuous nonlocal mean curvature. Left: plot of $u_h$ in $[-1.5,1.5]$, right: plot of $u_h$ near origin.} \label{fig:1d_discontinuity_inside}
\end{figure}

As a second illustration of interior discontinuities, let $\Omega = (-1,1)^2 \subset \mR^2$, $s=0.01$, $g = 0$ and consider $f(x,y) = 4 \, \textrm{sign}(xy)$. We use a mesh graded toward the axis and boundary with $N = 4145$ degrees of freedom and plot the numerical solution $u_h$ in \Cref{fig:2d_checkerboard}. The behavior of $u_h$ shows that the solution $u$ has discontinuities near the boundary and across the edges inside $\Omega$ where $f$ is discontinuous. 

	\begin{figure}[!htb]
	\begin{center}
		\includegraphics[width=0.5\linewidth]{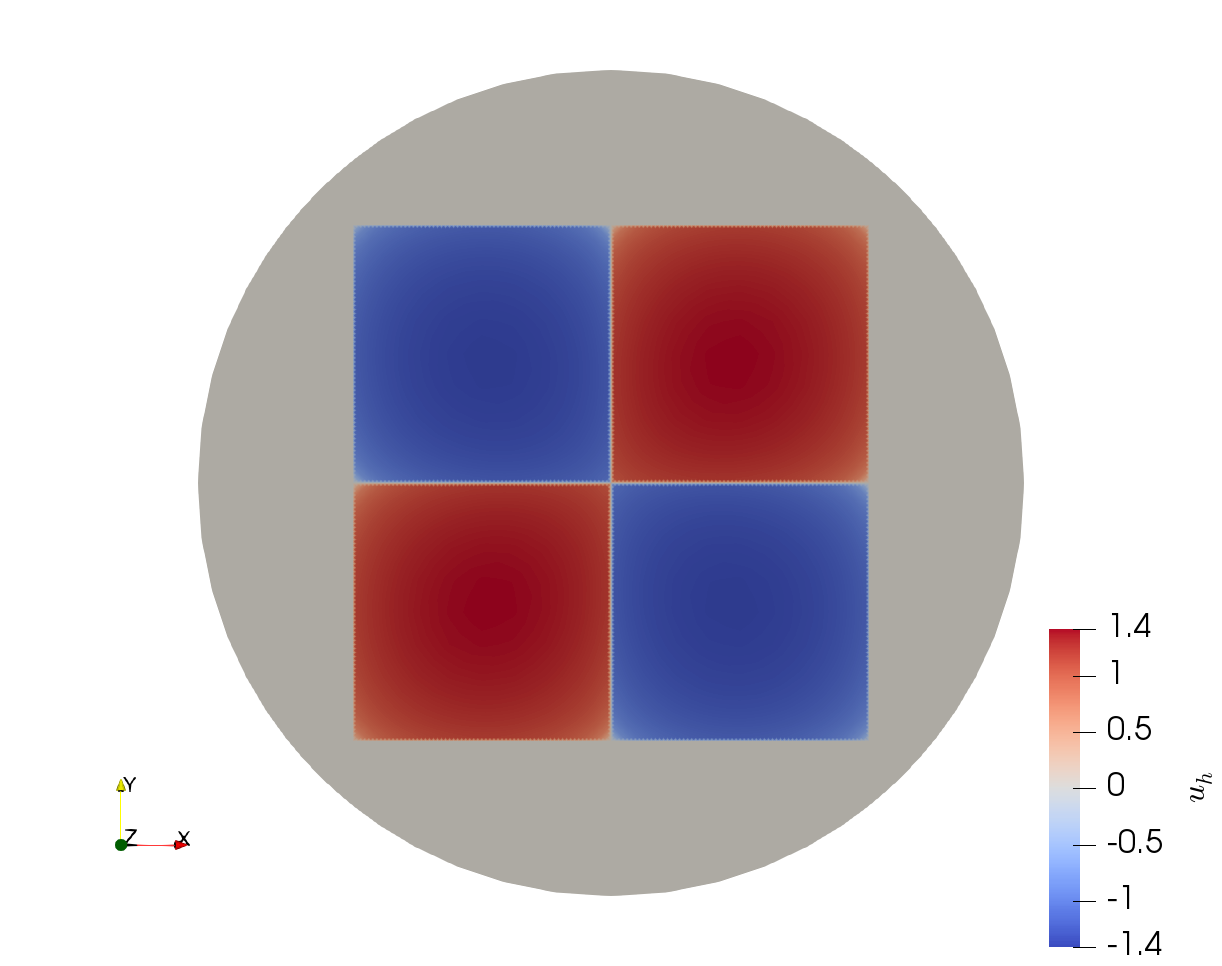}
		\includegraphics[width=0.35\linewidth]{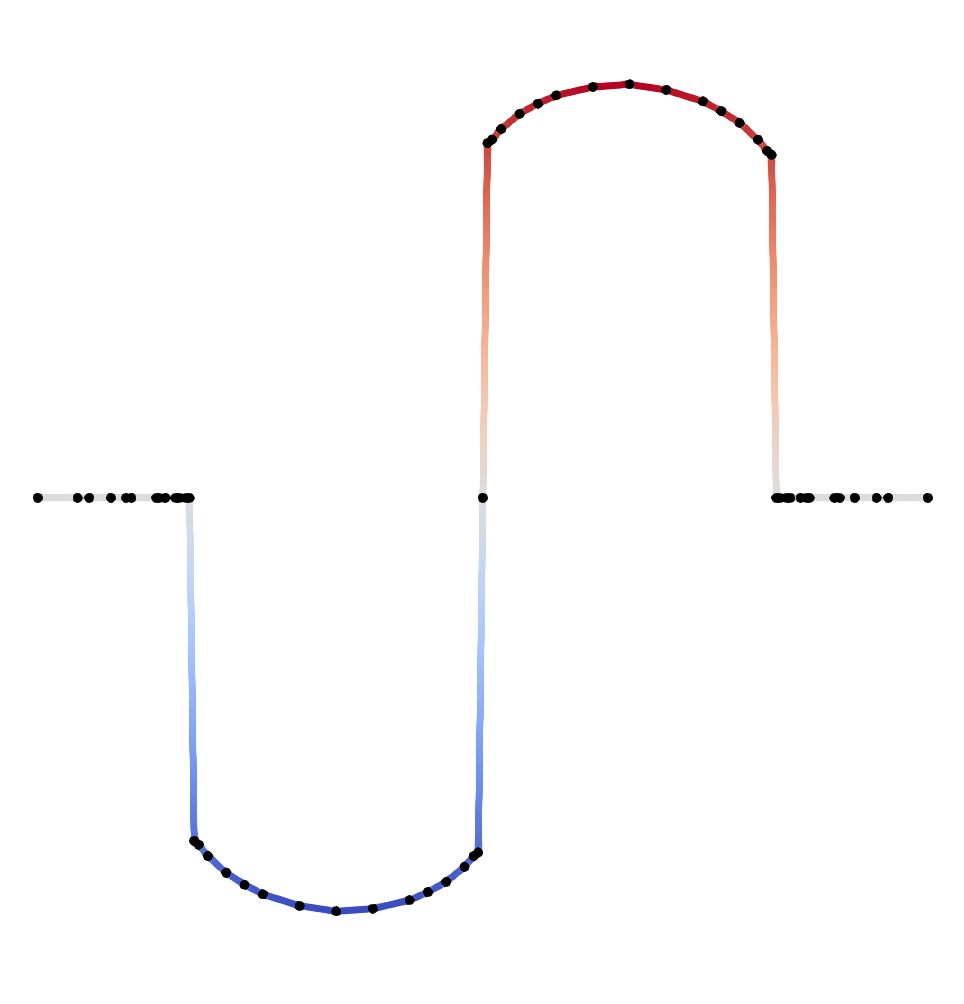}
	\end{center}
	\caption{
	A graph with prescribed discontinuous nonlocal mean curvature in the square $\Omega = (-1,1)^2$. The left panel displays a top view, while the right panel shows a side view along the slice $\{ y = 1/2\}$. The solution to \eqref{E:NMS-Prescribe-Nmc} is discontinuous inside $\Omega$. 	
} \label{fig:2d_checkerboard}
\end{figure}

%
%
%

\subsubsection{Effect of boundary curvature}
Next, we numerically address the effect of boundary curvature over nonlocal minimal graphs. For this purpose, we present examples of graphs with prescribed nonlocal mean curvature in several two-dimensional domains, in which we fix $g = 0$ and $f = -1$.

Consider the annulus $\Omega = B(0,1) \setminus B(0,1/4)$ and $s=0.25$. The top row in \Cref{fig:curvature_square} offers a top view of the discrete solution $u_h$ and a radial slice of it. We observe that the discrete solution is about three times stickier in the inner boundary than in the outer one. The middle and bottom row in \Cref{fig:curvature_square} display different views of the solution in the square $\Omega = (-1,1)^2$ for $s = 0.01$. Near the boundary of the domain $\Omega$, we observe a steep slope in the middle of the edges; however, stickiness is not observed at the convex corners of $\Omega$.

\begin{figure}[!htb]
	\begin{center}
		\begin{tabular}{c c} 
		\includegraphics[width=0.42\linewidth]{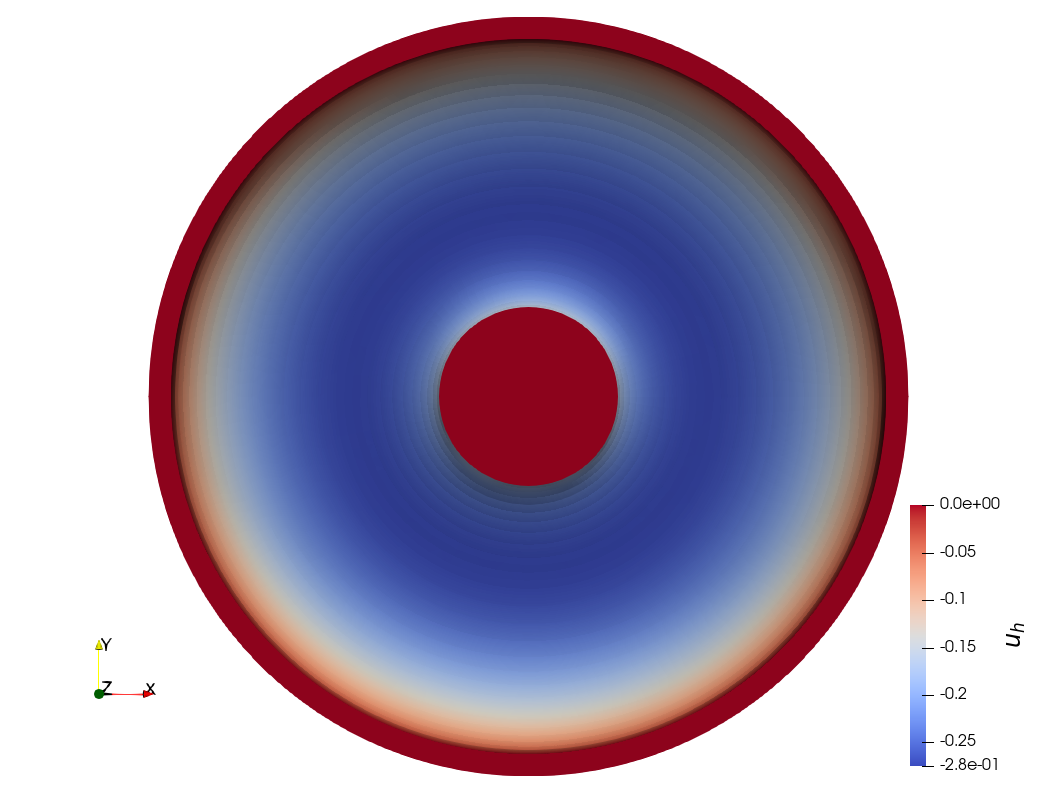}& 
		\includegraphics[width=0.42\linewidth]{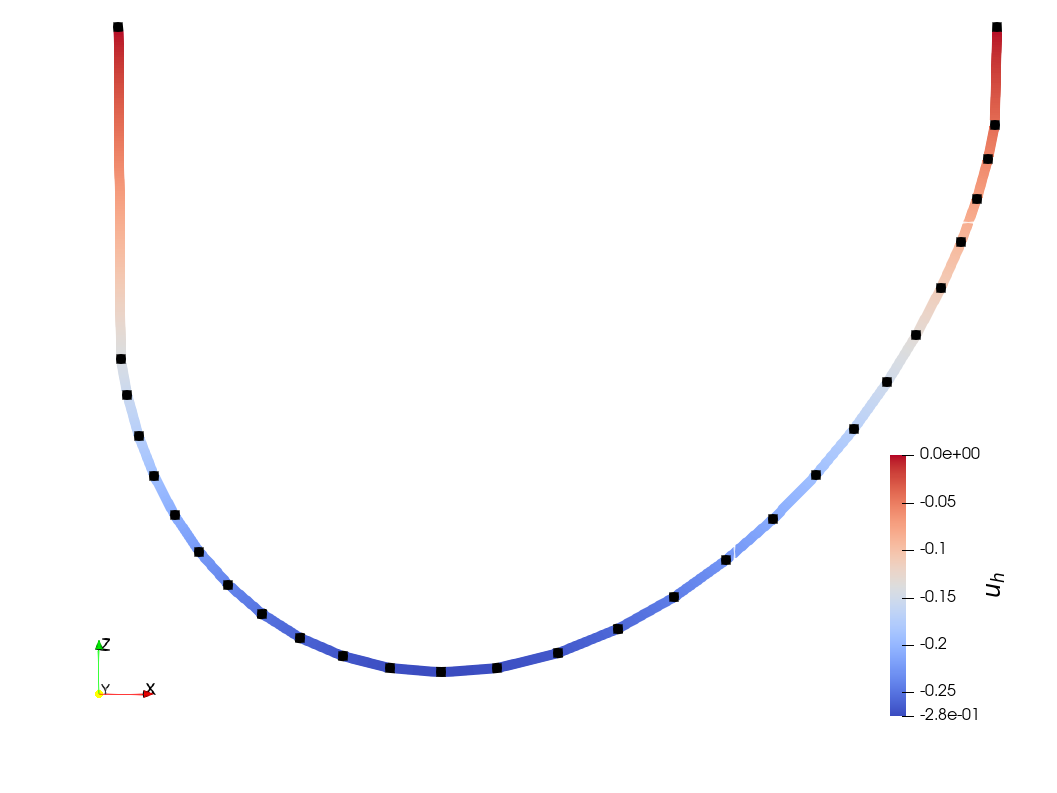}\\
		\includegraphics[width=0.42\linewidth]{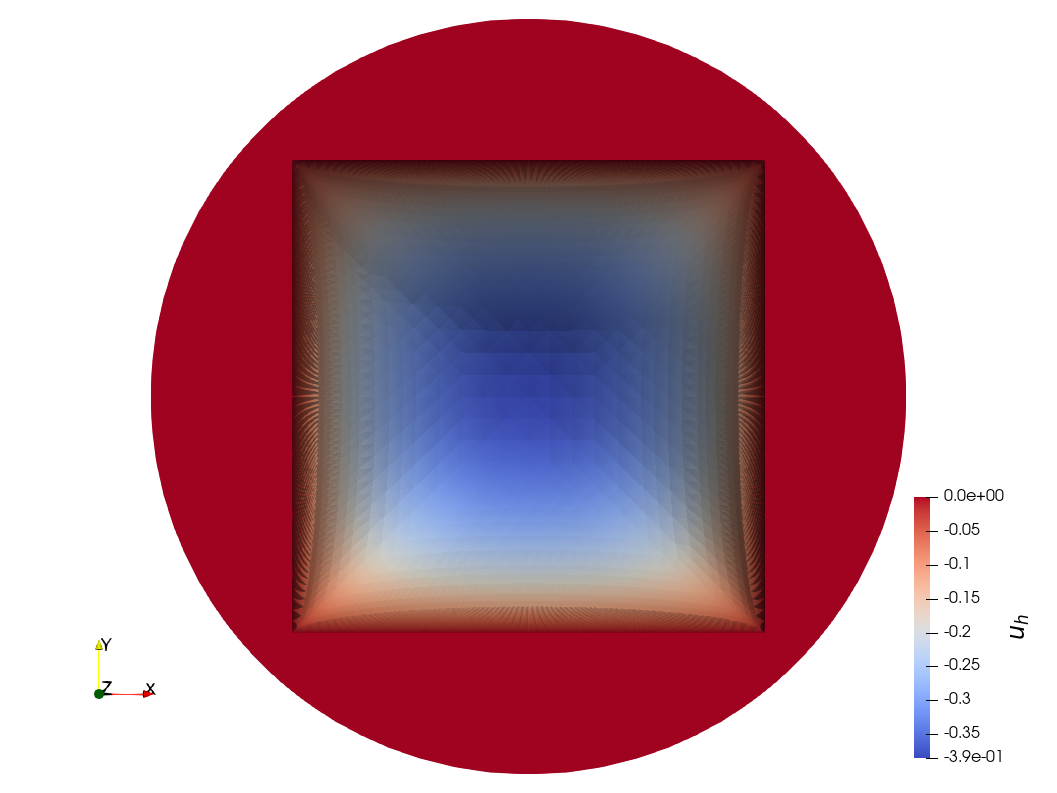} 
&
		\includegraphics[width=0.42\linewidth]{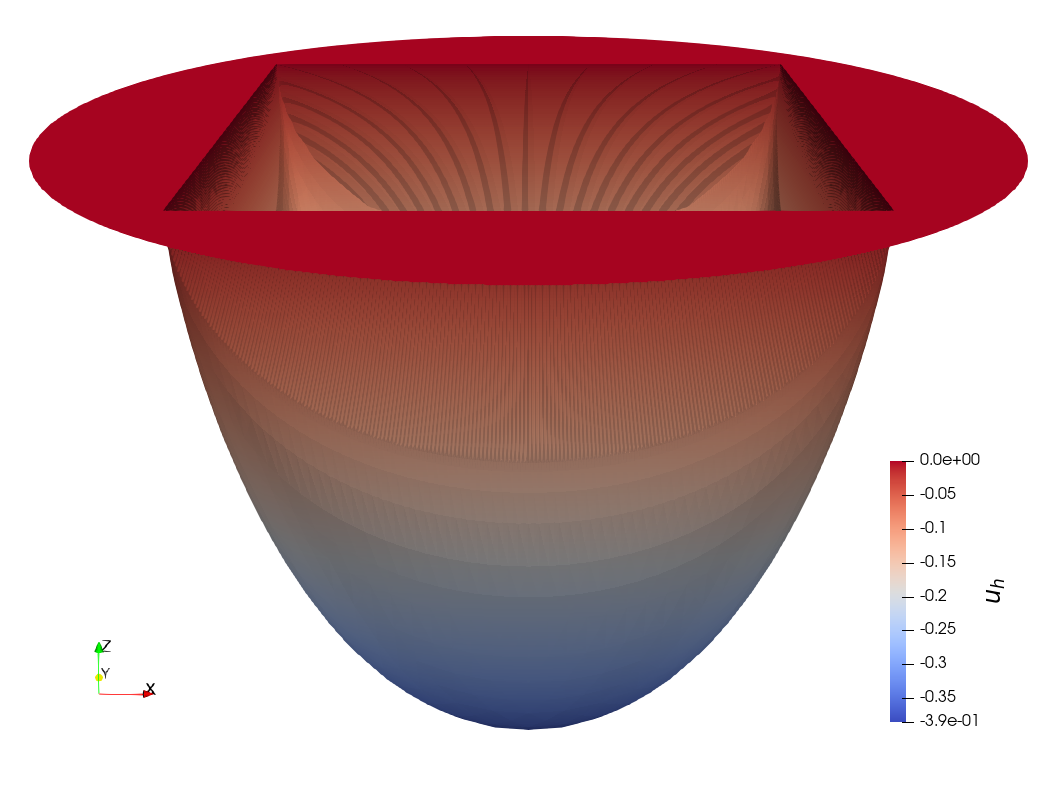} \\
		\includegraphics[width=0.42\linewidth]{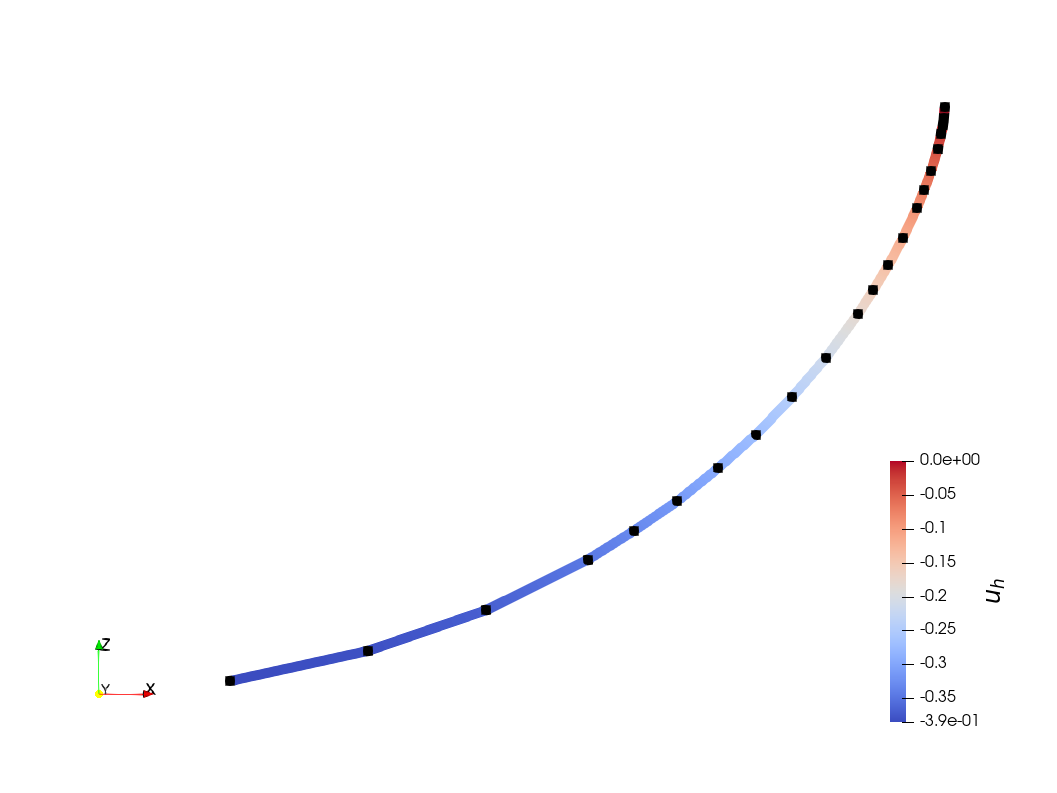} 	&	\includegraphics[width=0.42\linewidth]{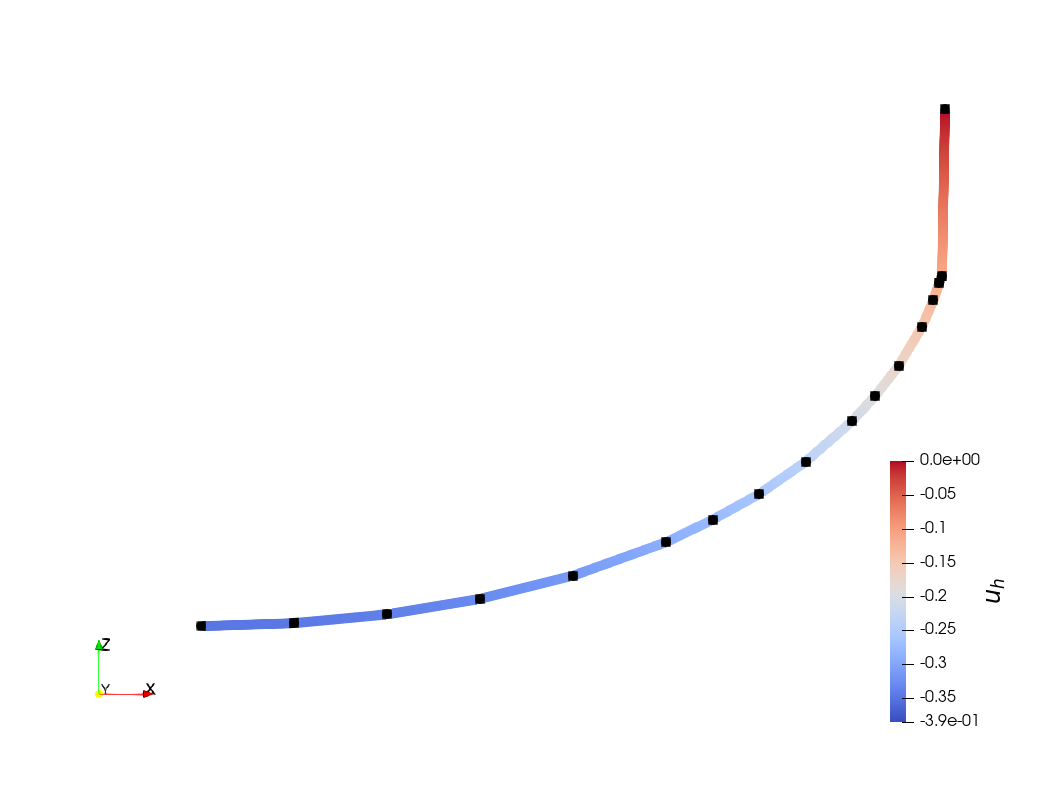}
		\end{tabular}
	\end{center}
	\caption{ Top and side views of functions with prescribed fractional mean curvature $f = -1$ in $\Omega$ that vanish in $\Omega^c$. Here, $\Omega$ is either an annulus (top row) or a square (middle and bottom row). The plot in the top-right panel corresponds to a radial slice ($y = 0, \,\,  0.25 \le x \le 1$) of the annulus, while the ones in the bottom-left and bottom-right show slices along the diagonal ($0 \le y = x \le 1$) and perpendicular to an edge of the square ($y = 0.5, 0 \le x \le 1$), respectively. We observe that stickiness is larger near the concave portions of the boundary than near the convex ones, and that it is absent in the corners of the square.
} \label{fig:curvature_square}
\end{figure}

We finally investigate stickiness at the boundary of the L-shaped domain $\Omega = (-1,1)^2 \setminus (0,1)\times (-1,0)$ with $s = 0.25, g = 0, f = -1$. We observe in \Cref{fig:curvature_Lshape} that stickiness is most pronounced at the reentrant corner but absent at the convex corners of $\Omega$. 

\begin{figure}[!htb]
	\begin{center}
		\begin{tabular}{c c} 
			\includegraphics[width=0.42\linewidth]{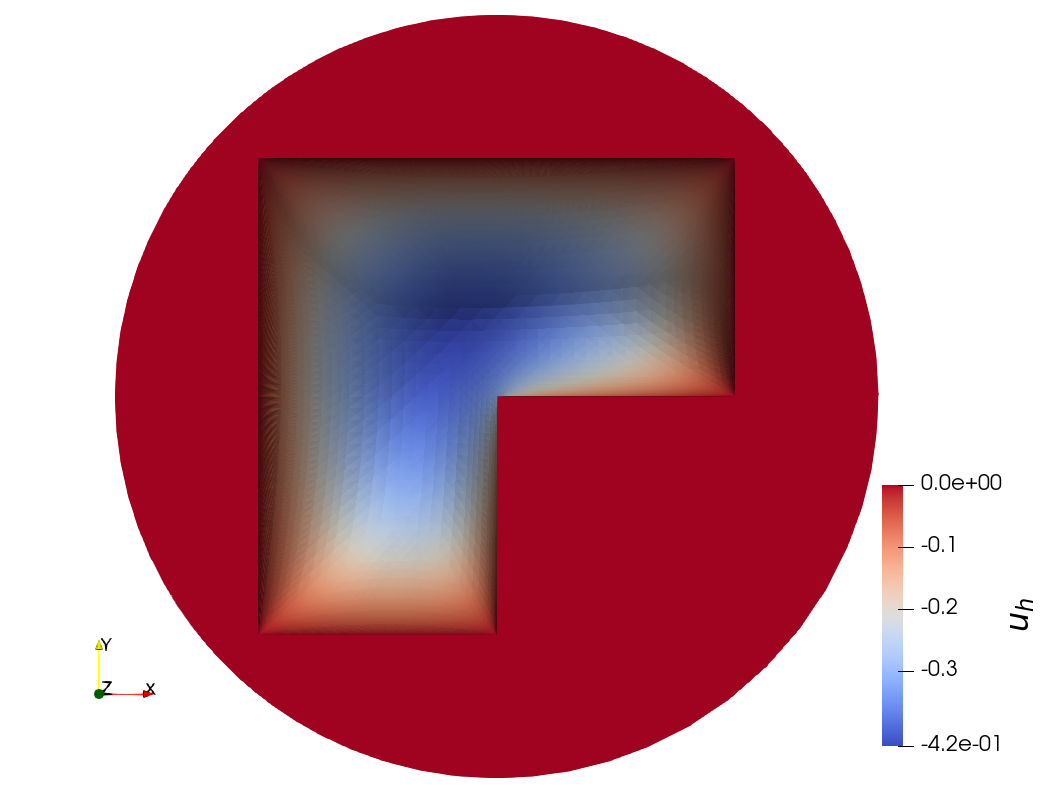}& 
			\includegraphics[width=0.42\linewidth]{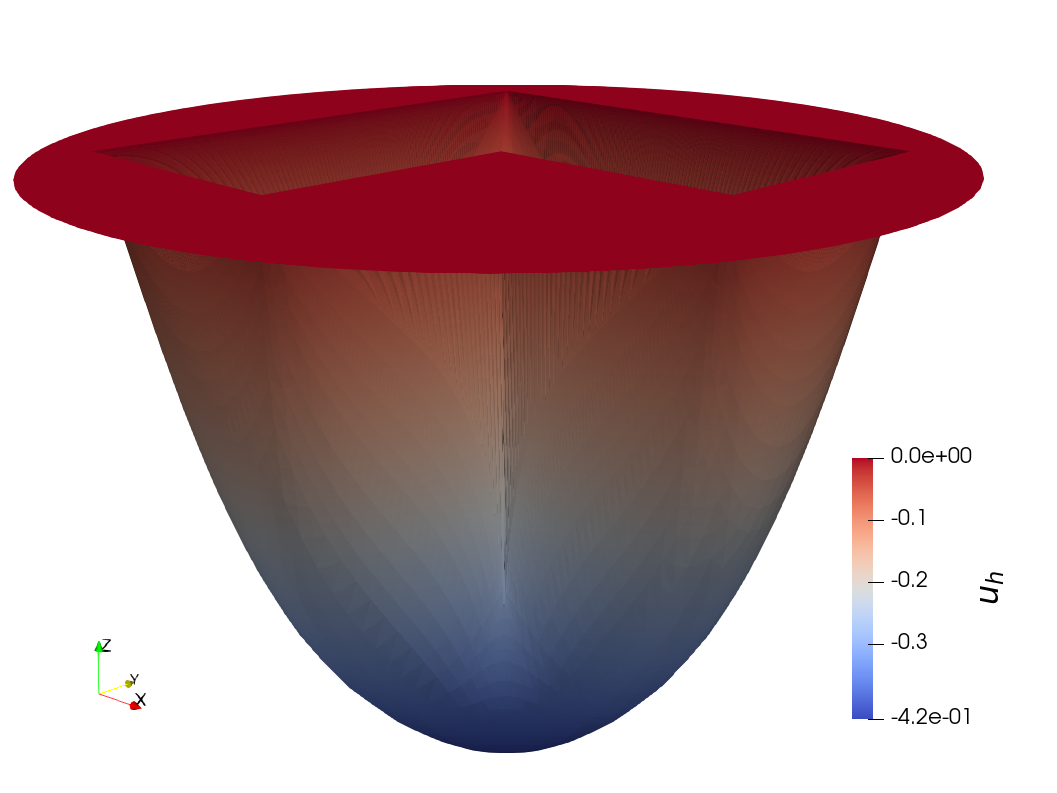}\\
			\includegraphics[width=0.42\linewidth]{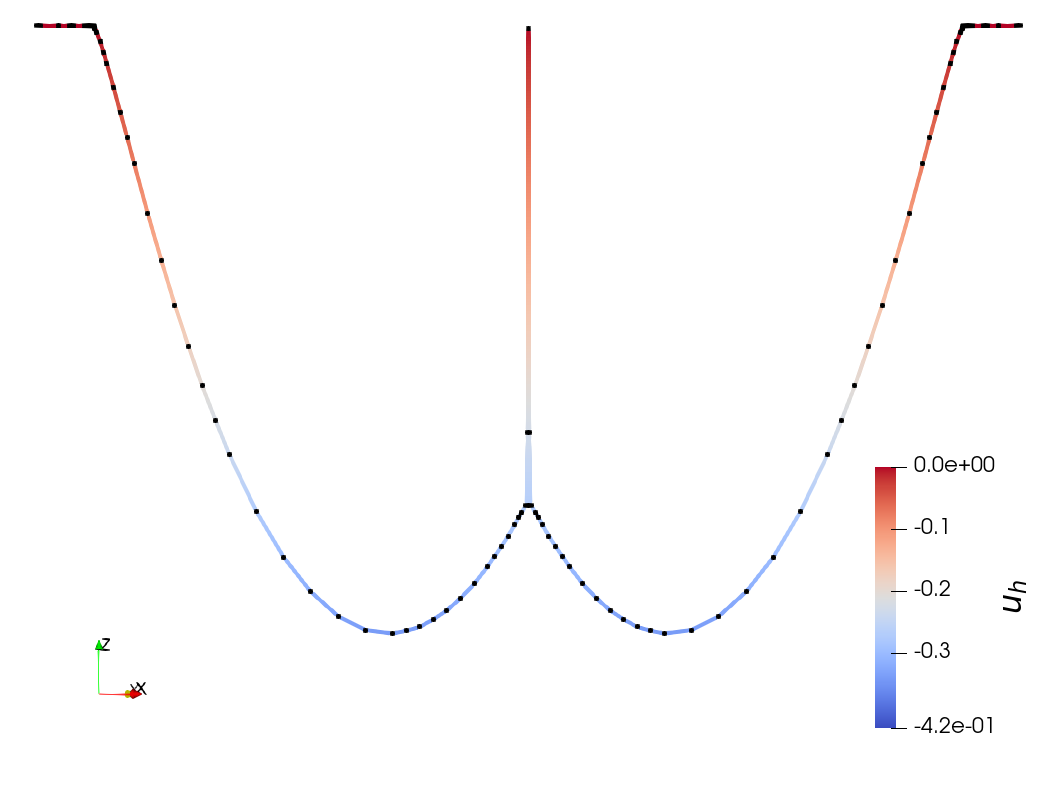} 
			&
			\includegraphics[width=0.42\linewidth]{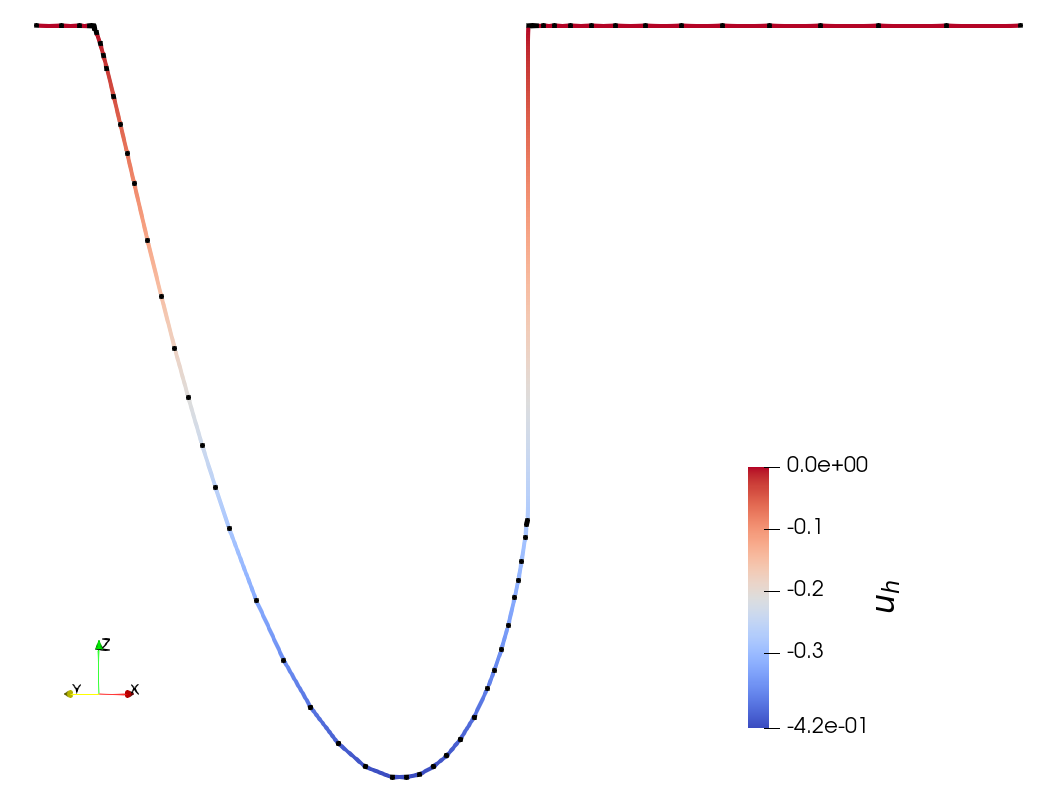} \\
			\includegraphics[width=0.42\linewidth]{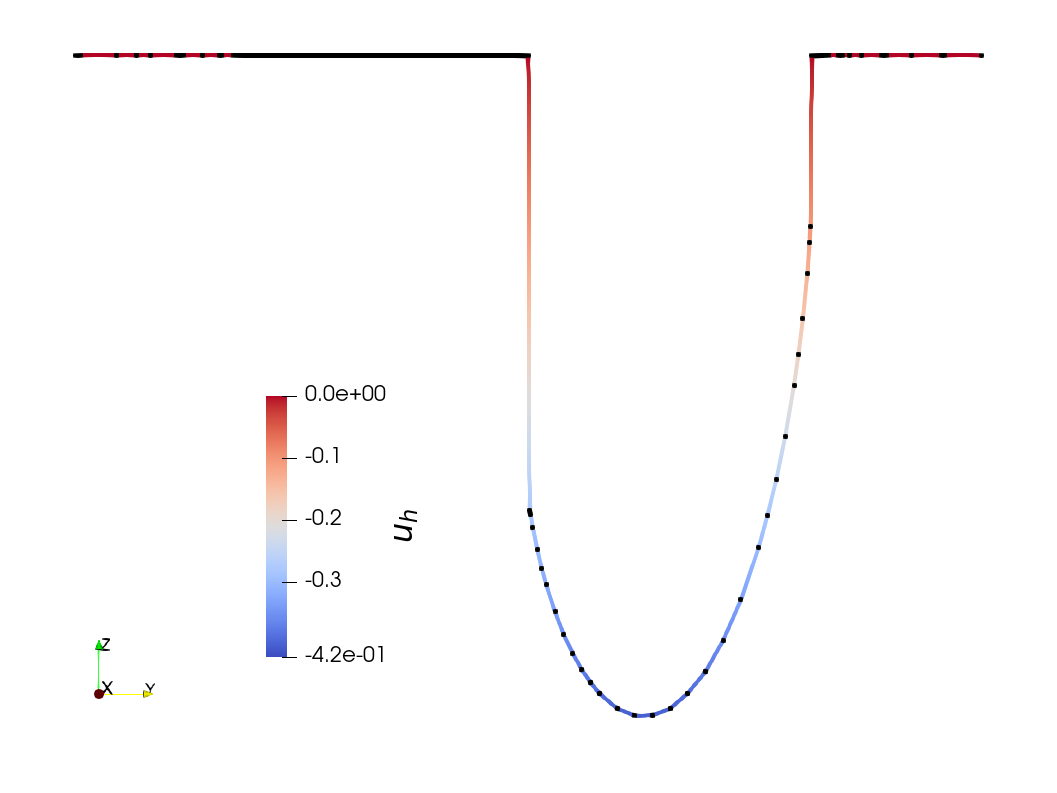} 	&	\includegraphics[width=0.42\linewidth]{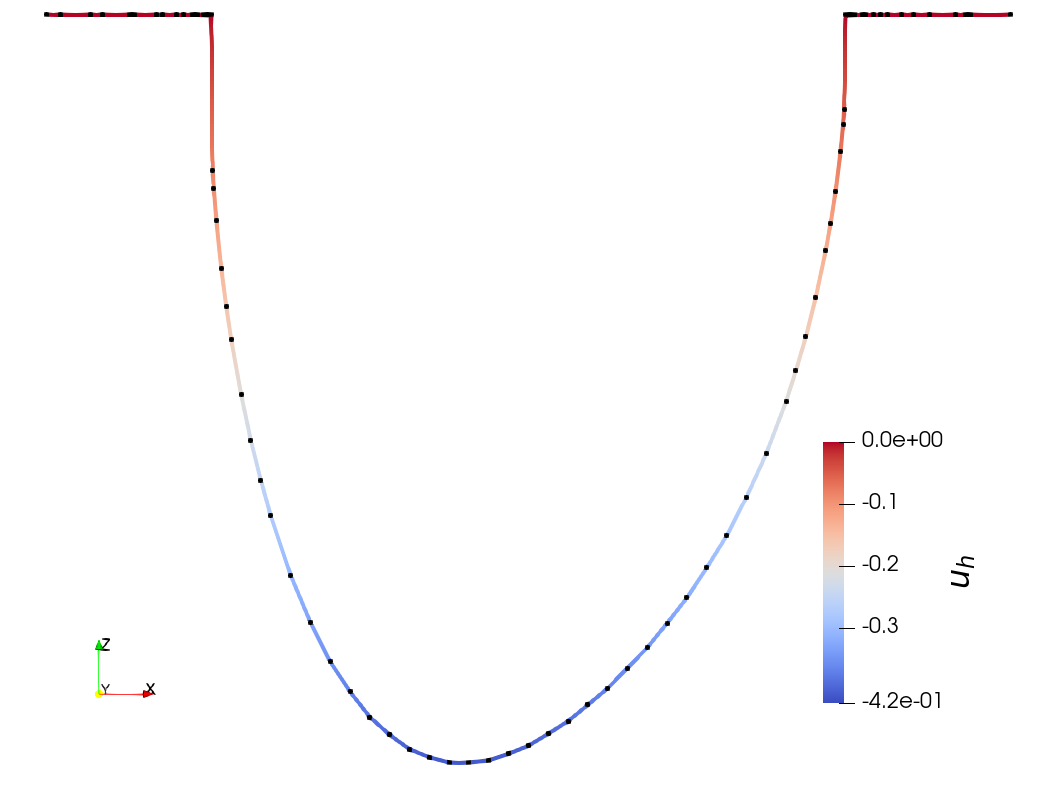}
		\end{tabular}
	\end{center}
	\caption{Stickiness on the L-shaped domain $\Omega = (-1,1)^2 \setminus (0,1)\times (-1,0)$ with prescribed fractional mean curvature $f=-1$ in $\Omega$ and Dirichlet condition $g=0$ in $\Omega^c$. The plots in the middle correspond to slices along $y = x$ and $y = -x$ respectively, while the ones in the bottom are slices along $x = 0$ or $y = 0.5$ respectively. We see that the largest stickiness takes place at the reentrant corner while there is no stickiness at the convex corners.} 
	\label{fig:curvature_Lshape}
\end{figure}

From these examples we conjecture that there is a relation between the amount of stickiness on $\partial\Omega$ and the nonlocal mean curvature of $\partial\Omega$.
Heuristically, let us assume that the Euler-Lagrange equation is satisfied at some point $x \in \partial \Omega$:
\[
H_s[u](x) = \PV \int_{\mR^d} G_s\l(\frac{u(x)-u(y)}{|x-y|}\r) \frac{dy}{|x-y|^{d+2s}} = f(x),
\]
where we recall that $G_s$ is defined in \eqref{E:DEF-Gs}.
This fact is not necessarily true, because \eqref{E:NMS-Prescribe-Nmc} guarantees this identity to hold on $\Omega$ only. Above, we assume that the minimizer is continuous in $\overline{\Omega}$, so that we can set $u(x) := \lim_{\Omega \ni y \to x} u(y)$. Thus, we can define the stickiness at $x \in \partial \Omega$ as
\[
M_s(x) := \lim_{\Omega^c \ni y \to x} u(y) - u(x).
\]
We point out that in these examples, because the minimizer $u$ attains its maximum on $\Omega^c$ and is constant in that region, we have $M_s \ge 0$. Let $r>0$ be small, and let us assume that the prescribed curvature is $f(x) = 0$, that we can split the principal value integral in the definition of $H_s$ and that the contribution of the integral on $\mRd \setminus B_r(x)$ is negligible compared with that on $B_r(x)$. Then, we must have 
\begin{equation*}\label{eq:stickiness-balance}
\int_{\Omega \cap B_r(x)} G_s\l(\frac{u(x)-u(y)}{|x-y|}\r) \frac{dy}{|x-y|^{d+2s}} \approx
\int_{\Omega^c \cap B_r(x)} G_s\l(\frac{u(y) - u(x)}{|x-y|}\r) \frac{dy}{|x-y|^{d+2s}}.
\end{equation*}
If the solution is sticky at $x$, namely $M_s > 0$, then we can approximate
\[
\int_{\Omega^c \cap B_r(x)} G_s\l(\frac{u(y) - u(x)}{|x-y|}\r) \frac{dy}{|x-y|^{d+2s}} \approx \int_{\Omega^c \cap B_r(x)} G_s\l(\frac{M_s}{|x-y|}\r) \frac{dy}{|x-y|^{d+2s}}.
\]

Due to the fact that $G_s\l(\frac{M_s}{|x-y|}\r)$ is strictly increasing with respect to $M_s$, we can heuristically argue that stickiness $M_s(x)$ grows with the increase of the ratio
\[
R(x) := \frac{|\Omega\cap B_r(x)|}{|\Omega^c\cap B_r(x)|}
\]
in order to maintain the balance between the integral in $\Omega \cap B_r(x)$ with the one in $\Omega^c \cap B_r(x)$. Actually, if $R(x)<1$, as happens at convex corners $x\in\partial\Omega$, it might not be possible for these integrals to balance unless $M_s(x) = 0$.
This supports the conjecture that the minimizers are not sticky at convex corners.

\section{Concluding remarks}
This paper discusses finite element discretizations of the fractional Plateau and the prescribed fractional mean curvature problems of order $s \in (0,1/2)$ on bounded domains $\Omega$ subject to exterior data being a subgraph. Both of these can be interpreted as energy minimization problems in spaces closely related to $W^{2s}_1(\Omega)$.

We discuss two converging approaches for computing discrete minimizers: a semi-implicit gradient flow scheme and a damped Newton method. Both of these algorithms require the computation of a matrix related to weighted linear fractional diffusion problems of order $s + \frac{1}{2}$. We employ the latter for computations.

A salient feature of nonlocal minimal graphs is their stickiness, namely that they are generically discontinuous across the domain boundary. Because our theoretical results do not require meshes to be quasi-uniform, we resort to graded meshes to better capture this phenomenon. Although the discrete spaces consist of continuous functions, our experiments in \Cref{sec:stickiness_thm2} show the method's capability of accurately estimating the jump of solutions across the boundary. In \Cref{sec:behavior_2d} we illustrate a geometric rigidity result: wherever the nonlocal minimal graphs are continuous in the boundary of the domain, they must also {\em match the slope of the exterior data.}  Fractional minimal graphs may change their convexity within $\Omega$, as indicated by our experiments in \Cref{sec:convexity}.

The use of graded meshes gives rise to poor conditioning, which in turn affects the performance of iterative solvers. Our experimental findings reveal that using diagonal preconditioning alleviates this issue, particularly when the grading is not too strong. Preconditioning of the resulting linear systems is an open problem.

Because in practice it is not always feasible to exactly impose the Dirichlet condition on $\mRd \setminus \Omega$, we study the effect of data truncation, and show that the finite element minimizers $u_h^H$ computed on meshes $\Th$ over computational domains $\Omega_H$ converge to the minimal graphs as $h \to 0$, $H\to 0$ in $W^{2r}_1(\Omega)$ for $r \in [0,s)$. This is confirmed in our numerical experiments.

Our results extend to prescribed minimal curvature problems, in which one needs some assumptions on the given curvature $f$ in order to guarantee the existence of solutions. We present an example of an ill-posed problem due to data incompatibility. Furthermore, our computational results indicate that graphs with discontinuous prescribed mean curvature may be discontinuous in the interior of the domain. We explore the relation between the curvature of the domain and the amount of stickiness, observe that discrete solutions are stickier on concave boundaries than convex ones, and conjecture that they are continuous on convex corners.

\bibliographystyle{plain}
\bibliography{NMS}
\end{document}